\documentclass[12pt]{article}
\usepackage{amsmath,amsfonts,amstext,amssymb,amsbsy,amsopn,amsthm,eucal,dsfont,graphicx}

\newtheorem{theorem}{Theorem}[section]
\newtheorem{prop}{Proposition}[section]
\newtheorem{lemma}{Lemma}[section]

\theoremstyle{definition}\newtheorem{definition}{Definition}[section]
\theoremstyle{remark}\newtheorem{remark}{Remark}[section]
\theoremstyle{remark}\newtheorem{example}{Example}[section]
\theoremstyle{remark}\newtheorem{problem}{Problem}[section]

\begin{document}

\title{Geometric Structures of Collapsing Riemannian Manifolds II: $N^{*}$-bundles and Almost Ricci Flat Spaces}

\author{Aaron Naber\thanks{Department of Mathematics, Princeton University,
        Princeton, NJ 08544 ({\tt anaber@math.princeton.edu}).} and Gang Tian\thanks{Department of Mathematics, Princeton University, Princeton, NJ 08544 ({\tt tian@math.princeton.edu}).}}

\date{\today}
\maketitle
\begin{abstract}
In this paper we study collapsing sequences $M_{i}\stackrel{GH}{\rightarrow} X$ of Riemannian manifolds with curvature bounded or bounded away from a controlled subset. We introduce a structure over $X$ which in an appropriate sense is dual to the $N$-structure of Cheeger, Fukaya and Gromov.  As opposed to the $N$-structure, which live over the $M_{i}$ themselves, this structure lives over $X$ and allows for a convenient notion of global convergence as well as the appropriate background structure for doing analysis on $X$.  This structure is new even in the case of uniformly bounded curvature and as an application we give a generalization of Gromov's Almost Flat Theorem and prove new Ricci pinching theorems which extend those known in the noncollapsed setting.  There are also interesting topological consequences to the structure.
\end{abstract}

%%% ----------------------------------------------------------------------
\section{Introduction}

This paper is the second in a series meant to study geometric structures of collapsing manifolds with bounded curvature, except possibly on a controlled subset.  The purpose of this paper is two fold.  First in the context of collapsing $n$-dimensional manifolds $(M^{n}_{i},g_{i})\rightarrow X$ with uniformly bounded sectional curvature we show the existence of structure over $X$ that allows both for a convenient notion of smooth convergence of the $M_{i}$ to $X$ as well as a necessary foundation for doing analysis on $X$ which takes into account this collapsing process.  The motivation for this structure is the following oversimplified picture; we will generalize it in a moment.  Let us assume that the $(M_{i},g_{i})$ have uniformly bounded sectional curvature and for the moment that the limit space $X$ is a manifold such that the collapse is nil (technical assumptions that will not be needed in the end).  Then the claim is that there is a vector bundle $V\rightarrow X$ which in an appropriate sense is the limit of the tangent bundles $TM_{i}$.  More precisely if $f_{i}:M_{i}\rightarrow X$ are continuous Gromov Hausdorff maps (the existence of which are well known in this context) then after passing to a subsequence we have that the pullback bundles $f^{*}_{i}V$ over $M_{i}$ are vector bundle isomorphic to the tangent bundles $\phi_{i}:TM_{i}\rightarrow f^{*}_{i}V$.  First notice that this statement alone has some content, in that even though no two of the $M_{i}$ need be diffeomorphic or even homotopic and yet their tangent bundles are all pullbacks of the same fixed bundle.  More than this there is a fiber metric $h$ on $V$ such that if $\phi^{*}_{i}h$ is the induced Riemannian metric on $M_{i}$ then $||\phi^{*}_{i}h-g_{i}||_{C^{1,\alpha}}\rightarrow 0$, and hence we see that there is a global notion of smooth convergence behind the scenes.  It is important to note that since $rank(V)=n=dim(M_{i})$ that the bundle $V\rightarrow X$, which we view as the limit of the tangent bundles $TM_{i}$ of $M_{i}$, is not the tangent bundle $TX$ of $X$.  In fact there turns out to be a canonical decomposition $V\approx TX\oplus V^{ad,X}$, where $V^{ad,X}$ represents the part of the tangent bundles $TM_{i}$ which point in the collapsing directions.  Now the main assumption here is that we assumed $X$ was a manifold.  We can get around this in the spirit of \cite{F1} by studying the limit of the frame bundles $(FM_{i},g^{FM}_{i})\rightarrow (Y, g^{Y})$, where the limit is now always a manifold.  In this case what we said all goes through and what we end up with is an equivariant vector bundle $V^{T}\rightarrow Y\stackrel{O(n)}{\rightarrow} X$.  We will see from the construction of $V^{T}$ that sections of $V^{T}$ are in some sense dual to the elements of the $N$-structure sheaf, so we will refer to $V^{T}$ as the $N^{*}$-bundle.  The bundle $V^{T}\rightarrow X$ is extremely important for the analysis of $X$ as it captures much more information than is in the geometry of $X$ itself. As first applications we generalize Gromov's Almost Flat theorem to the Ricci situation and prove a new Ricci pinching theorem in Theorem \ref{cor_nt2_main1}.

The second purpose of the paper is motivated by the desire to understand the metric structure of limits of four manifolds with bounded Ricci and Euler number.  Recall that one of the main consequences in the first paper \cite{NaTi} in the context of bounded Ricci curvature was to see that if $(M^{4}_{i},g_{i}) \stackrel{GH}{\rightarrow} X$, where the $M_{i}$ have uniformly bounded Ricci and Euler characteristic, then away from a finite number of points $\{p_{j}\}^{N}_{1}\in X$ we have that $X_{NS}\equiv X-\{p_{j}\}$ is a Riemannian orbifold.  The question is what is the structure of $X$ near these points.  We will see that the bundle $V^{ad,X}$ mentioned above is canonically flat and in the next paper that the holonomy of this flat connection is directly related to a removable singularity question at these points.

To describe more precisely the construction of the $N^{*}$-bundle we give a brief and simplified overview of the origins of the $N$-structure.  Again begin with the case that $(M_{i},g_{i})$ have uniformly bounded curvature and let $(M^{n}_{i},g_{i})\stackrel{GH}{\rightarrow} (X,d)$.  Even assume for now that $X$ is a manifold.  Then by the work of Fukaya \cite{F} for large $i$ the $M_{i}$ take the form of fiber bundles $M_{i}\stackrel{f_{i}}{\rightarrow}X$ whose fibers are infranil manifolds $N_{i}/\Lambda_{i}$, where the $N_{i}$ are simply connected nilpotent Lie Groups and $\Lambda_{i}< N_{i}\rtimes Aut(N_{i})$ are discrete lattices.  Thus if $U\subseteq X$ is a small ball then $f^{-1}_{i}(U)\approx U\times (N_{i}/\Lambda_{i})\equiv U_{i}$ and if we look at the universal cover $\widetilde{f^{-1}_{i}(U)}\approx U\times N_{i}\equiv \widetilde{U_{i}}$ then there is an action by $N_{i}$ on this universal cover.  Because the fibers have small diameter the subgroup $\Lambda_{i}$ is acting in an increasingly dense fashion on each fiber and using the sectional curvature bounds one can see that $N_{i}$ itself acts almost isometrically.  If we view $N_{i}$ as a right action then the derivative of this action gives left invariant vertical vector fields which are almost Killing fields.  At least locally these pass back down to vector fields on $M_{i}$, and if one is careful these local vector fields obtained from each such $U\subseteq X$ can be used throughout $M_{i}$ to give the global sheaf.  More generally because the frame bundles $FM_{i}\rightarrow Y$ always collapse to a manifold \cite{F1} this can at least always be done on the $FM_{i}$ in an equivariant manner.  Hence even when $X$ is not a manifold such a sheaf may be built on the $M_{i}$, though now the dimension of the orbits of various points may change.

The $N^{*}$-bundle is constructed by instead of considering the left invariant vertical vector fields with respect to this local $N_{i}$ action on $\widetilde{U_{i}}$ we consider the local right invariant vector fields.  If we restrict to right invariant vertical vectors these would locally correspond to sections of the local adjoint bundle $U\times\eta_{i}\rightarrow U$ where $\eta_{i}$ is the lie algebra of $N_{i}$.  More generally it is convenient to consider all right invariant vectors, which locally corresponds to sections of $TU\times\eta_{i}$.  Again if one is careful then it turns out that these local vector bundles may be pasted together into global vector bundles (at least as long as the fibers are nil), the later of which we call the invariant tangent bundle $V_{i}^{T}\rightarrow X$ and the former we call the adjoint bundle $V_{i}^{ad}\rightarrow X$.  We remark that since the $N_{i}$ action is local one cannot view $V_{i}^{ad}$ as an actual adjoint bundle, however in a generalized sense the name is appropriate because of the local construction of the bundle.  That $V_{i}^{T}\approx V^{T}$ and $V_{i}^{ad}\approx V^{ad}$ are actually independent of $i$ as vector bundles is an issue dealt with in Section \ref{sec_V_str} (though the lie algebra structures of the bundles can change in the limit!) .  To put a geometry on $V^{T}$ we note that if the $g_{i}$ were invariant under the local $N_{i}$ action on each $\widetilde{U_{i}}\approx U\times N_{i}$ that in this case we see that $g_{i}$ induce fiber metrics $g_{i}^{T}$ on the invariant tangent bundle $V^{T}\rightarrow X$.  It is clear from the construction that such a fiber metric $g_{i}^{T}$ also uniquely determines an $N_{i}$-invariant metric $g_{i}$ on $X$ (see subsequent paragraphs for more details on this).  In terms of the adjoint bundle $V^{ad}\rightarrow X$ we see that an $N$-invariant metric $g_{i}$ determines an affine connection $\nabla^{ad}$ and fiber metric $h^{ad}$ on $V^{ad}$.  Again in the case where $X$ is not a manifold these constructions follow through on the frame bundles in an equivariant fashion, and one can use this to construction to obtain an equivariant vector bundle and geometric data over the frame limit $Y\rightarrow X$.

A precise definition of the $N^{*}$-bundle over a space $X$ is as follows, where it is assumed $G$ is a compact Lie Group (it will be an orthogonal group in practice).

\begin{definition}\label{def_vstr}
Let $X$ be a topological space.  An $N^{*}$-bundle over $X$ is a smooth $G$-manifold $Y$ with finite principal isotropy called the frame space together with a $G$-vector bundle $V^{T}\rightarrow Y$, a surjective $G$-mapping $\rho: V^{T}\rightarrow TY$, a $G$-invariant fiber metric $g^{T}$ and a nilpotent lie algebra $\eta$ such that:
\begin{enumerate}

\item There exists a covering $\{U_{\alpha}\}$ of $Y$ such that for each $U\in\{U_{\alpha}\}$ the restriction $V|_{U}\rightarrow U$ has local trivialization as the bundle $TU\times\eta$ so that $\rho:TU\times\eta\rightarrow TU$ is locally just the projection to the first coordinate.

\item There exists a $G$-invariant flat connection $\nabla^{flat}$ on the $G$-bundle $V^{ad}\equiv ker\rho$ such that in the above trivializations the connection is the trivial one.  Further for two coverings $U_{0}\cap U_{1}\neq \emptyset$ as above the induced coordinate transformation on $\eta$ is an affine transformation.

\item $Y/G$ is homeomorphic to $X$ and the $G$-action on $V^{T}$ induces lie algebra homomorphisms on $ker\rho$.
\end{enumerate}
\end{definition}

The fiber metric $g^{T}$ allows for a horizontal inverse map $\rho^{-1}:TY\rightarrow V^{T}$ which we may use to pull back $g^{T}$ to get a $G$-invariant Riemannian metric $g^{Y}$ on $Y$ and hence an induced length space geometry $d^{X}$ on $X$.  Conditions $\ref{def_vstr}.1$ and $\ref{def_vstr}.2$ guarantee specific reductions on the structure group of $V$, while condition $\ref{def_vstr}.3$ states that the $G$ action is compatible with these reductions.  See Section \ref{sec_flat_conn} for even further reductions of the structure group in practice.

For each $U\in\{U_{\alpha}\}$ if $N$ is the simply connected nilpotent lie group associated to $\eta$ let us see how to use the $N^{*}$-bundle to construct a Riemannian manifold $(U\times N,g^{U\times N})$.  The metric $g^{U\times N}$ should be invariant under the right $N$ action with the property that $g^{U\times N}/N =  g^{Y}$.  In practice if $M_{i}\stackrel{f_{i}}{\rightarrow} Y$ then for nicely chosen Gromov Hausdorff maps $f_{i}$ (see Section \ref{sec_estimates}) the space $(U\times N,g^{U\times N})$ should represent the Cheeger-Gromov limit of the universal covers $(\widetilde{f^{-1}_{i}(U),g_{i}})$.  These universal covers do in fact have injectivity radius bounds and so such a limit is possible.  Thus we will see how the $N^{*}$-bundle recaptures the unwrapped limits of $M_{i}$.

So let $U\in\{U_{\alpha}\}$ and $N$ as before.  Let $\xi^{a}$ be a vector basis of $TU$ and $\zeta^{\hat a}$ be a basis of $\eta$.  Then $\{\xi^{a},\zeta^{\hat a}\}$ naturally forms a vector field basis for $TU\times\eta$.  The canonical coordinate association $T(U\times N)\approx TU\times TN$ gives us a natural isomorphism $T(U\times N)/N\approx TU\times\eta$.  So we can also view $\{\xi^{a},\zeta^{\hat a}\}$ as an $N$- invariant vector field basis of $T(U\times N)$ and the fiber metric $g^{V}$ induces a right invariant metric $g^{U\times N}$ on $U\times N$ as desired.  It is also worth pointing out that $g^{V}$ alone does not describe the geometry (i.e. curvature) of $U\times N$.  For that we also need to know the brackets of the basis $\{\xi^{a},\zeta^{\hat a}\}$, which is precisely what the lie algebra $\eta$ gives us.  The $G$-equivariance allows a similar procedure to recapture the geometry on $X$.  See Section \ref{sec_geom_V_ad} for more on that.

Now the complete (nonlocal) picture is as follows, below we give the formal local theorem.  Let $(M^{n}_{i},g_{i}) \stackrel{GH}{\rightarrow} (X,d)$ where the $M_{i}$ are complete with uniformly bounded curvature for simplicity.  Then as in \cite{F1} and Lemma \ref{lem_frame} we see that we can put $O(n)$-metrics $g^{FM}_{i}$ on the frame bundles $FM_{i}$ such that $(FM_{i},g^{FM}_{i},O(n))\stackrel{eGH}{\rightarrow} (Y,g^{Y},O(n))$ where $(Y,g^{Y})$ is a Riemannian manifold with $(Y,g^{Y})/O(n)\approx (X,d)$.  However $(Y,g^{Y})$ loses considerable information about the collapsing sequence itself.  We claim there is a $N^{*}$-bundle $V^{T}\rightarrow Y\rightarrow X$ above $X$ such that, after passing to a subsequence, we can pick the equivariant Gromov Hausdorff maps $f_{i}:FM_{i}\rightarrow Y$ so that there exist equivariant vector bundle isomorphisms $\phi_{i}:TFM_{i}\rightarrow f^{*}_{i}V^{T}$ such that $||(\phi_{i}\circ f_{i} )^{*}g^{V}-g^{FM}_{i}||_{C^{1,\alpha}}\rightarrow 0$.  We will also see in Section \ref{sec_geom_V_ad} that the geometry of the $N^{*}$-bundle more completely describes the structure of the metric on the collapsed space.  An analysis of the unwrapped limits has also been done by Lott using a groupoid approach in \cite{Lo}, where it is applied to understand limits of Ricci flows.

To state the theorem carefully we need a couple technical definitions.  Primarily this is because we need to have these constructions be purely local, which add a lot of notational mess without greatly altering the total picture.  We refer to Section \ref{sec_notation} for the definitions involving metric regularity but essentially we just need to distinguish between the more or less equivalent bounds for a Riemannian manifold $(M,g)$ in form of curvature bounds (which we refer to as regular bounds in the spirit of \cite{CFG}) and the existence of good weak coordinates (which we refer to as bounded geometry, see Section \ref{sec_notation}).  We call a $N^{*}$-bundle $(\{A\}^{k+1,\alpha}_{0},r)$-bounded if the induced metric $g^{Y}$ is $(\{A\}^{k+1,\alpha}_{0},r)$-bounded and one can pick coordinates as in definition \ref{def_vstr} such that the fiber metric $g^{T}$ is $(\{A\}^{k+1,\alpha}_{0},r)$-bounded (see Section \ref{sec_notation}).

As mentioned the construction of the $N^{*}$-bundle takes place on the frame bundles $FM_{i}$ of a collapsing sequence.  Recall as in \cite{F1} that if $(M,g)$ is a Riemannian manifold then there exists a metric $g^{FM}$ on $FM$ such that $(FM,g_{FM})/O(n)\approx (M,g)$.  Even better we can build this metric as in Lemma \ref{lem_frame} to have the same bounded regularity properties as $(M,g)$, and the metric from this lemma is the one we will always use when we refer to a geometry on the frame bundle.

Note as in Section \ref{sec_notation} that if $M$ has boundary then we define $M_{\iota} \equiv \{x\in M: d(x,\partial M)>\iota\}$.  If $M$ has no boundary then $M_{\iota} \equiv M$.  If $M_{i}\stackrel{GH}{\rightarrow}X$ then we define $X_{\iota}$ as the Gromov Hausdorff limit of $M_{i,\iota}$. We state the following for compact limits $X$.  However since the $M_{i}$ are allowed to have boundary the same statements hold for arbitrary $X$, we simply need to replace \textit{convergence} with \textit{converence on compact subsets}.

\begin{theorem}\label{thm_nt2_main1}
Let $(M_{i}^{n},g_{i})\stackrel{GH}{\rightarrow} (X,d)$ where the $M_{i}$ are $\{A\}^{k}_{0}$-regular spaces and $(X,d)$ is a compact metric space.  Then for each $\iota>0$ and $0<\alpha<1$ we have after passing to a subsequence:
\begin{enumerate}
\item There exists an $(\{B\}^{k+1,\alpha}_{0},r)=(\{B(n,A)\}^{k+1,\alpha}_{0},r(n,A^{0}))$-bounded $N^{*}$-bundle $V^{T}\rightarrow Y\stackrel{O(n)}{\rightarrow}X_{\iota}$.

\item For each $i$ there exists $O(n)$-equivariant Gromov Hausdorff maps $f_{i}:FM_{i,\iota}\rightarrow Y$ which induce the convergence $(FM_{i},g^{FM}_{i},O(n))\stackrel{eGH}{\rightarrow}(Y,g^{Y},O(n))$.

\item There exists $O(n)$-equivariant vector bundle isomorphisms $\phi_{i}:TFM_{i}\rightarrow f_{i}^{*}V^{T}$ such that $||(\phi_{i}\circ f_{i})^{*} g^{V}-g^{FM}_{i}||_{C^{k+1,\alpha}}\rightarrow 0$.
\end{enumerate}
\end{theorem}

\begin{remark}
It is worth mentioning that what is required above is that the geometry of the $M_{i}$ be $(\{A\}^{k+1,\alpha}_{0},r)$-bounded for $k\geq -1$.  The same statement holds under this assumption.
\end{remark}

We would like to point out that bundles similar to $V^{T}$ and $V^{ad}$ have appeared previously in the literature, namely in \cite{Lo1} and \cite{Lo3}.  They were used to analyze limits of Dirac operators on spaces with bounded curvature.  A finiteness theorem similar to the results of Section \ref{sec_V_str} is proved there as well, though the estimates of Section \ref{sec_estimates} are needed for Theorem \ref{thm_nt2_main1}.  The structure of Section \ref{sec_flat_conn} is also not provided, and this is key to the proof of the next theorem.

Now we wish to give some brief applications of a more thorough analysis of this bundle.  As motivation for the next result consider the following simplified version, which can be found in \cite{F3}.  Assume $(M^{n}_{i},g_{i})\stackrel{GH}{\rightarrow}(X,d)$ with $diam(M_{i})=1$ and $|sec_{i}|\rightarrow 0$.  Then it follows, and the proof can be considered a generalization of the Bieberbach theorem, that $(X,d)$ is a flat Riemannian orbifold.  The moral here is that while upper sectional bounds do not generally pass to the limit, in the case of pinching they must.  We prove a similar statement for the Ricci case below, though the proof is more involved.

\begin{theorem}\label{thm_nt2_main2}
Let $(M^{n}_{i},g_{i})\stackrel{GH}{\rightarrow}(X,d)$ where the $M_{i}$ are complete with $diam(M_{i})=1$, $|sec_{i}|\leq K$ and $|Rc_{i}|\rightarrow 0$.  Then $(X,d)$ is a Ricci flat Riemannian orbifold with $|sec|\leq K$.
\end{theorem}

\begin{remark}
We could replace the $|sec|\leq K$ assumption with $conj(M)\geq K^{-1}>0$, where $conj$ is the conjugacy radius of the exponential map.  This follows because as in \cite{A},\cite{PWY} we then have a $\{A\}^{1,\alpha}_{0}$-bounded geometry control over $M$.  In this case $|sec_{X}|\leq C=C(n,K)$.  It will also follow from the proof that if we had only assumed $Rc_{i}\leq \epsilon_{i}\rightarrow 0$ then $X$ still has at worst orbifold singularities.
\end{remark}

One could also attempt to prove the above using a groupoid approach.  In this context one would first need to construct a sequence of nested subgroupoid structures similar to the subbundles of $V^{ad}$ in the limit central decomposition in Section \ref{sec_flat_conn}.  Further the holonomy of the flat connection in Section \ref{sec_flat_conn} plays a crucial role in the proof of the above theorem and an analogue of that must be constructed.  Though it seems reasonable that constructions similar to those of the paper could be made in the groupoid context, because of the role of the flat connection of $V^{ad}$ both directly and in the construction of the limit central decomposition it feels more natural to the authors to work in the category of vector bundles.

Now let us first see that the above is sharp, in that if we remove any of the hypothesis then the result fails to be true.  In particular if we remove the completeness or sectional bound assumption then the result must fail, so that the proof must be highly global in nature.  In fact the proof will be through a series of maximum principles on $X$ and the analysis involved will depend crucially on the $N^{*}$-bundle structure we have been discussing.  First let us see some examples, the first of which is due to Gross and Wilson \cite{GW}.

\begin{example}
There exists Ricci flat metrics $(K3,g_{i})$ on a $K3$ such that $(K3,g_{i})\stackrel{GH}{\rightarrow}(X,d)$, where $X$ is a topological two sphere and the distance function $d$ is induced from a Riemannian metric $g_{\infty}$ on $X$ which is smooth away from $24$ points, where the curvature blows up.  In particular $g_{\infty}$ is not a Ricci flat orbifold metric.
\end{example}

The above tells us two things.  First the sectional bound in Theorem \ref{thm_nt2_main2} is quite necessary.  Secondly by picking an open set $U\subseteq X$ away from the 24 singular points and considering the subsets $U_{i}\subseteq(K3,g_{i})$ which converge to $U$ we see that the completeness assumption is necessary, because in this case it is not too hard to check that on $U_{i}$ all the other conditions remain valid but $U$ is certainly not Ricci flat.  In fact the following example tells us the diameter assumption is necessary as well:

\begin{example}
Consider the space $\mathds{C}\times S^{1}$ with the natural flat metric $g$ and the $S^{1}$ action by $\lambda\cdot(z,\bar\lambda)=(\lambda z,\lambda\bar\lambda)$, where we have identified $S^{1}$ with the unit complex numbers.  Then we can consider the sequence $(\mathds{C}\times S^{1},g)/\mathds{Z}_{p}\stackrel{p\rightarrow\infty}{\rightarrow} (\mathds{C},g_{\infty})$.  Then the limit $(\mathds{C},g_{\infty})$ has strictly positive sectional curvature and in particular is not Ricci flat.
\end{example}

As a corollary of Theorem \ref{thm_nt2_main2} we get the following generalization of Gromov's Almost Flat theorem.  Gromov's theorem proves that there is a dimensional constant $\epsilon(n)$ such that if $diam(M^{n})=1$ and $|sec|\leq\epsilon$ then $M$ is topologically an infranil.  That is to say that though $M$ may not be a flat space that up to a finite cover $M$ is a sequence of fiber bundles of flat spaces over flat spaces.  In the case where $|Rc|$ is small this is too much to ask, but we see below that $M$ is now a sequence of fiber bundles of flat spaces over a Ricci flat space.

\begin{theorem}\label{cor_nt2_main1}
Let $(M^{n},g)$ be complete with $|sec|\leq K$ and $diam=1$.  Then there exists an $\epsilon=\epsilon(n,K)>0$ such that if $|Rc|\leq\epsilon$ then $M$ is an orbifold bundle over a Ricci flat orbifold $(X,d)$ with $|sec_{X}|\leq K$ whose fibers are infranil.
\end{theorem}
\begin{remark}
As in the last theorem instead of the assumption $|sec|\leq K$ we may assume $conj\geq K^{-1}>0$.  It is also worth pointing out that if $T$ is the second fundamental form of the infranil fibers then for each $\delta>0$, arbitrarily small, one can pick $\epsilon=\epsilon(n,K,\delta)$ so that $|T|<\delta$.  This of course is far from true in general collapsing.
\end{remark}

Note that the above can also be viewed as a generalization of a pinching result of Anderson's \cite{An}. In Anderson's theorem, where one makes the noncollapsing assumption $inj\geq K^{-1}>0$ instead of $conj\geq K^{-1}$, then $M$ admits a Ricci flat metric.  In the collapsed scenario this is simply not the case, but the statement is that there does exist a nearby Ricci flat space $X$ that $M$ fibers over.  It even follows from the proof of the above that $X$ has a uniform curvature bound of $K$ (in the case where $conj\geq K^{-1}$ then $X$ has an \textit{apriori} curvature bound depending only on $n$ and $K$).

The organization of the paper is as follows.  In Section \ref{sec_notation} we begin with some notation.  This section can primarily be referred back to as needed.  Then in Section \ref{sec_estimates} we improve some estimates from \cite{CFG} on smoothed Gromov Hausdorff maps between Riemannian manifolds.  Namely in \cite{CFG} it is proved that if $f:(M_{0},g_{0})\rightarrow (M_{1},g_{1})$ is an $\epsilon$-Gromov Hausdorff map then we can modify $f$ by a center of mass technique to produce a smooth map which is a $\sqrt{\epsilon}$-GH map while having uniform $C^{2}$ bounds (independent of $\epsilon$) and being an almost Riemannian submersion.  Their technique cannot be modified to prove higher order bounds, which we will need, and so this is handled in Section \ref{sec_estimates}.

Section \ref{sec_framespace} studies and constructs various structures on the limit $Y$ of the frame bundles $FM_{i}$ of a collapsing sequence.  As mentioned it is convenient to have metrics $g^{FM}_{i}$ on the frame bundles that do not have the regularity issues that come from the canonical geometric construction of \cite{F1}, so first we handle this.  Then using the results of Section \ref{sec_estimates} and the constructions from \cite{CFG} we build \textsl{good} coordinates on $FM_{i}\rightarrow Y$.  Appendix \ref{sec_inf_bund} reviews some of the constructions of \cite{CFG} which we use, as well as some mild generalizations.

Once the technical constructions of Section \ref{sec_framespace} are complete we can begin to build the $N^{*}$-bundle.  This is done in two steps.  In Section \ref{sec_approx_V_str} we build approximate $N^{*}$-bundles for each $FM_{i}$.  This construction is the rigorous version of what has been outlined throughout this section.  The key point of the $N^{*}$-bundle though is that it is independent of $i$, and so in Section \ref{sec_V_str} we show how these approximate structures limit in an appropriately strong sense and use the limit to prove Theorem \ref{thm_nt2_main1}.

In Section \ref{sec_flat_conn} we mention some refinements of the construction of the canonical flat connection on $V^{ad}$.  This connection and its refinements are particularly important when the $M_{i}$ only have curvature bounded away from a controlled subset.  We will also use it in Section \ref{sec_thm2proof} when we prove Theorem \ref{thm_nt2_main2}.  Then in Section \ref{sec_geom_V_ad} we study limit geometrical structures on $V^{ad}$, which is also used in Section \ref{sec_thm2proof} to finish the proof of Theorem \ref{thm_nt2_main2}.

\section{Notation}\label{sec_notation}

Due to the unfortunate amount of notation required this section is meant to organize some of the various definitions which are used throughout, it can mainly referred back to as need be.  Much is standard, self explanatory and can be found in other sources, however a little is new and there are occasionally mild modifications to older notation.

Often we will be dealing with manifolds with boundary so we define

\begin{definition}
Let $(M,g)$ be a Riemannian manifold, possibly with boundary.  For all $\iota>0$ we define $M_{\iota} \equiv \{x\in M: d(x,\partial M)>\iota\}$.  If $M$ has no boundary then $M_{\iota} \equiv M$.  Let $M^{\circ}$ denote the interior of $M$.
\end{definition}

Also because we often deal with manifolds with boundary it is important to be clear by what we mean by the injectivity radius of the manifold.

\begin{definition}
Let $(M,g)$ be a Riemannian manifold, possibly with boundary.  For $x\in M$ let $inj_{M}(x)\leq d(x,\partial M)$ be the standard injectivity radius with $inj_{M}(x) = d(x,\partial M)$ being the well defined statement that the exponential map is a diffeomorphism up to the boundary.  Define the boundary injectivity radius $inj^{B}(M)\equiv sup\{\iota>0: \forall x\in M$ $inj_{M}(x)\geq min(\iota,d(x,\partial M)) \}$.
\end{definition}

In fact for the purposes of these results there are many ways one could define a boundary injectivity radius that would be equally suitable.  The importance is simply that once you push away from the boundary there is control of the exponential map.

There are two basic but essentially equivalent notions of having controlled geometry that we will be interested in.  The first is that of bounded curvature and the second is that of the existence of good (weak) coordinates.  For analysis purposes it is the second that is usually the useful one to deal with, while the first is more intrinsic and useful for the statement of theorems.  So to begin with we follow \cite{CFG} and introduce the following

\begin{definition}
Let $(M^{n},g,p)$ be a pointed Riemannian manifold possibly with boundary and $\{A\}^{k}_{0}$ a sequence of $k+1$ positive real numbers.  We say $(M,g,p)$ is $\{A\}^{k}_{0}$-regular (resp. interior regular) at $p$ if $\forall$ $r <\pi A_{0}^{-1/2}$ $B_{r}(p)$ has compact closure in $M$ (resp. $M^{\circ}$) with $|sec_{g}|\leq A^{0}$ and $|\nabla^{(j)} Rm_{g}|\leq A^{j}$ $\forall$ $1\leq j \leq k$ in $B_{r}(p)$.  We say $(M,g)$ is $\{A\}^{k}_{0}$-regular if it is $\{A\}^{k}_{0}$-regular at every point.
\end{definition}

The second notion is more useful while doing analysis and taking limits and is the notion of bounded geometry.  First we adopt the notation of \cite{PWY} and introduce the following, where $\mathds{R}^{n,+}$ denotes the closed upper half plane in $\mathds{R}^{n}$.

\begin{definition}
Let $M^{n}$ be a smooth manifold, then we call a mapping $\varphi: B_{r}(0) \subseteq \mathds{R}^{n} \rightarrow M$ (resp. $\varphi: B_{r}(0) \subseteq \mathds{R}^{n,+} \rightarrow M$ if $\varphi(0)$ is a boundary point) a weak coordinate system if $\varphi$ is a local diffeomorphism.  If $(M^{n},g)$ is Riemannian and $\{A\}^{k,\alpha}_{0}$ is a sequence of $k+1$ positive numbers with $0\leq\alpha<1$ and $r>0$ then we say $\varphi$ is $(\{A\}^{k,\alpha}_{0},r)$-bounded if, identifying $g$ with its pullback $\varphi^{*}g$, $e^{-\frac{A^{0}}{10}}\delta_{ij} \leq g_{ij} \leq e^{\frac{A^{0}}{10}}\delta_{ij}$ and for any multi-index $\mathbf{a}$ with $0\leq |\mathbf{a}| \leq k$ we have $r^{|\mathbf{a}|+\alpha}|| \partial_{\mathbf{a}}g_{ij}||_{C^{\alpha}} \leq A^{|\mathbf{a}|,\alpha}$.  We call $\varphi$ a weak harmonic coordinate system if locally $\varphi^{-1}:U\subseteq M \rightarrow \mathds{R}^{n}$ is a harmonic map.
\end{definition}

The only distinction between the $C^{k,\alpha}_{r}$ norm of $\varphi$ and the usual $C^{k,\alpha}$ norm considering $\varphi$ as a mapping from one Riemannian manifold into another is the $r$ weight.  The usual example of a weak coordinate system is to take $\varphi$ as the exponential map at a point when the conjugacy radius is strictly larger than the injectivity radius.  Now using such weak coordinate systems there is a natural notion of bounded geometry on a Riemannian Manifold:

\begin{definition}\label{def_boundedgeometry}
Let $(M,g)$ be a Riemannian manifold, $r>0$ and $\{A\}^{k,\alpha}_{0}$ be $k+1$ positive real numbers.  We say $(M,g)$ has $(\{A\}^{k,\alpha}_{0},r)$-bounded geometry at $x\in M$ if there exists a $(\{A\}^{k,\alpha}_{0},r)$-bounded weak coordinate system $\varphi:B_{r}(0)\rightarrow M$ with $\varphi(0)=x$.  We say $(M,g)$ has $(\{A\}^{k,\alpha}_{0},r)$-bounded geometry if this holds at each $x\in M$.  We say $(M,g)$ has harmonic bounded geometry if additionally $\varphi$ are harmonic weak coordinates.
\end{definition}

The following is a classic result, see \cite{T} and \cite{CH}:

\begin{lemma}\label{lem_weak_harm_coor}
Let $(M^{n},g)$ be $A^{0}$-interior regular at $x$.  Then for each $0\leq\alpha <1$ there exists $r_{h}=r_{h}(n,A^{0},\alpha)$ such that there exists a weak harmonic coordinate system $\varphi: B_{r_{h}}(0)\rightarrow M$ with $\varphi(0)=x$ and $||\varphi||_{C^{1,\alpha}_{r_{h}}}\leq 1$.
\end{lemma}

The above lemma essentially turns a $\{A\}^{k}_{0}$-regular condition into a  $(\{B\}^{k+1,\alpha}_{0},r)$-bounded condition.  More generally it will be important to write vector bundles over a Riemannian manifold in geometrically useful coordinates.  To this end we define

\begin{definition}
Let $(M,g)$ be a Riemannian manifold with $V\rightarrow M$ a vector bundle of rank $l$.  We call a vector bundle map $\varphi^{V}:B_{r}(0)\times\mathds{R}^{l}\rightarrow V$ a weak coordinate system if it is also a local diffeomorphism.  If $h$ is a fiber metric on $V$ we say $\varphi^{V}$ is $(\{A\}^{k,\alpha}_{0},r)$-bounded if the induced map $\varphi:B_{r}\rightarrow M$ is $(\{A\}^{k,\alpha}_{0},r)$-bounded such that in coordinates $e^{-\frac{A^{0}}{10}}\delta_{\hat i \hat j} \leq h_{\hat i \hat j} \leq e^{\frac{A^{0}}{10}}\delta_{\hat i \hat j}$ and for any multi-index $\mathbf{a}$ with $0\leq |\mathbf{a}| \leq k$ we have $r^{|\mathbf{a}|+\alpha}|| \partial_{\mathbf{a}}h_{\hat i \hat j}||_{C^{\alpha}} \leq A^{|\mathbf{a}|,\alpha}$.
\end{definition}

It is now straightforward to generalize weak coordinate systems from local diffeomorphisms from open sets in $\mathds{R}^{n}$ into our manifolds to local diffeomorphisms from other spaces with special properties into our manifolds. We will want to consider geometries on Lie groups which are appropriately nice:

\begin{definition}
Let $G$ be a Lie Group and $h$ a right invariant metric on $G$.  We define the norm of $h$ as $||h||\equiv ||ad||_{h}$, where $ad:\mathfrak{g}\times\mathfrak{g}\rightarrow\mathfrak{g}$ is the adjoint operator.  We say $h$ is normalized if $||h||\leq 1$.
\end{definition}

Notice that the above gives a natural way of taking limits of simply connected Lie Groups of bounded dimension.  Namely if $(\mathfrak g_{i},h_{i})$ is a sequence of $n$-dimenional lie algebras with normalized metrics $h_{i}$ and adjoint operators $ad_{i}$, then by identifying an orthonormal basis in $\mathfrak g_{i}$ with a standard basis in $\mathds{R}^{n}$ and passing to a subsequence we see that the $ad_{i}$ convergence to an adjoint operator $ad_{\infty}$ on $\mathds{R}^{n}$.  The Jacobi identity of the induced lie algebra still holds in the limit and so $ad_{\infty}$ itself defines a lie algebra $\mathfrak g_{\infty}$.  Note that even if $\mathfrak g_{i}\approx \mathfrak g$ are all isomorphic as lie algebras that $\mathfrak g_{\infty}$ may not be isomorphic to $\mathfrak g$ (it may be more abelian in fact).

We will need one more definition from this section for the paper.

\begin{definition}
Let $N$ be a simply connected nilpotent Lie Group.  We call a map $\varphi: B_{r}(0)\times N \rightarrow M$ a weak nilpotent coordinate system if $\varphi$ is a local diffeomorphism such that there exists a discrete subgroup $\Lambda \leq N\rtimes Aut(N)$ with the property that $\varphi(x\cdot\lambda)=\varphi(x) $ $\forall x\in B_{r}\times N$ and $\lambda\in\Lambda$ and the induced map $\varphi:B_{r}\times(N/\Lambda)\rightarrow M$ is a diffeomorphism onto its image.  If $(M,g)$ is Riemannian and $(N,h)$ is normalized we say $\varphi$ is $(\{A\}^{k,\alpha}_{0},r)$-bounded if when $B_{r}\times N$ is equipped with the natural product metric $\delta_{ij}+h_{ij}$ and we identify $g$ with its pullback metric on $B_{r}\times N$ then $e^{-\frac{A^{0}}{10}}(\delta+h)_{ij} \leq g_{ij} \leq e^{\frac{A^{0}}{10}}(\delta+h)_{ij}$ and for any multi-index $\mathbf{a}$ with $0\leq |\mathbf{a}| \leq k$ we have $r^{|\mathbf{a}|+\alpha}||\nabla_{\mathbf{a}}g_{ij}||_{C^{\alpha}} \leq A^{|\mathbf{a}|,\alpha}$.  The subgroup $\Lambda$ will sometimes be referred to as the covering group $\pi_{\varphi}$ of the mapping $\varphi$.
\end{definition}

\section{Higher order Gromov Hausdorff Estimates}\label{sec_estimates}

This section is dedicated to proving some technical estimates needed for the results.

We begin with some notation, in addition to which we will use the notation from Section \ref{sec_notation} and Appendix \ref{sec_inf_bund} frequently.  We will be interested in a slight modification of the usual notion of Gromov Hausdorff approximations:

\begin{definition}
Let $(X_{0},d_{0})$ and $(X_{1},d_{1})$ be complete metric spaces.  We say $f:X_{0}\rightarrow X_{1}$ is an $(r,\epsilon)$ Gromov Hausdorff approximation if
\begin{enumerate}
\item $f(X_{0})$ is $\epsilon$-dense in $X_{1}$

\item $\forall x,y\in X_{0}$ we have $|d_{0}(x,y)-d_{1}(f(x),f(y))|<\epsilon(\frac{1}{r}d(x,y)+1)$
\end{enumerate}
\end{definition}

Notice by the triangle inequality that condition $2)$ above is up to scale equivalent to insisting that $|d_{0}(x,y)-d_{1}(f(x),f(y))|<\epsilon$ for any $x,y\in X_{0}$ with $d(x,y)<r$.  Many of the constructions of this paper are local and as such it will be more convenient to work with $(r,\epsilon)$ GH approximations as opposed to the usual $\epsilon$ GH approximations, though fundamentally there is little difference for the purposes of this paper.  In fact it is worth pointing out that the results of this paper involving the $(\epsilon,r)$-GH distance can be also proven for the usual GH-distance function, however there is more technical work involved and we do not need it.  Of course we may also talk about $(r,\epsilon)$ $G$-equivariant GH approximations in the same sense, where $G$ is a compact Lie group.

Now if $(M_{0},g_{0})$ and $(M_{1},g_{1})$ are Riemannian manifolds and $f:M_{0}\rightarrow M_{1}$ is a submersion then we label $\mathcal{V}\subseteq TM_{0}$ and $\mathcal{H}\subseteq TM_{0}$ as the vertical and horizontal subbundles of $TM_{0}$, respectively. We call a horizontal vector field $X\in M_{0}$ basic if it is the horizontal lift of a vector field in $M_{1}$.

When $M_{0}$ and $M_{1}$ are $\epsilon$-close in the usual Gromov Hausdorff sense then one of the main technical tools of \cite{CFG} is the construction of a smooth map $f:M_{0}\rightarrow M_{1}$ between $M_{0}$ and $M_{1}$ which is a $\delta$-GH map such that this map is an almost Riemannian submersion in the $C^{1}$ sense (that is, $1-\delta\leq |df[X]| \leq 1+\delta$ for any unit horizontal vector $X$) with a uniform $C^{2}$ bound.  Higher derivative bounds are available but they necessarily degenerate as $\delta\rightarrow 0$ under the construction of \cite{CFG}.  Also it does not follow from their construction that $f$ is an almost Riemannian submersion in the $C^{2}$ sense, which is to say that although if $X$ is a unit horizontal vector then we know that $|\nabla^{2}f(X,X)|$ is bounded we would additionally like to be able to say that it is small.  The main goal of this section is to show this smooth map may be constructed to have sharp higher order derivative bounds and such that $f$ is an almost Riemannian submersion in the $C^{2}$ sense.  In this context sharpness of the derivative bounds means that if $M_{0}$ and $M_{1}$ are $\{A\}^{k}_{0}$-regular, then the map $f$ has uniform bounds on the first $k+2+\alpha$ derivatives which depend only on $\{A\}^{k}_{0}$ and some other necessary intrinsic data (but not $\epsilon$).

Throughout this section $G$ is taken to be a compact Lie group.  So in the spirit of \cite{CFG} and the above we define the following:

\begin{definition}
Let $(M_{0},g_{0})$ and $(M_{1},g_{1})$ be Riemannian $G$-manifolds with $\{B\}^{k,\alpha}_{1}$ positive real numbers with $k\in \mathds{N}$ and $0\leq\alpha<1$.  We say a smooth function $f:M_{0}\rightarrow M_{1}$ is a $\{B\}^{k,\alpha}_{1}$-bounded $(r,\epsilon)$-$G$-Gromov Hausdorff Approximation if
\begin{enumerate}
\item $f$ is an $(r,\epsilon)$ $G$-equivariant Gromov Hausdorff Approximation.

\item $diam(f^{-1}(y))< \epsilon$ for $\forall y\in M_{1}$

\item We have $\forall$ $X\in \mathcal{H}$ that $1-\epsilon\leq |df[X]|\leq 1+\epsilon$, and if $k\geq 2$ then additionally we have that $|\nabla^{2}f(X,X)|\leq B^{(2)}\epsilon$

\item $||f||_{C^{j,\alpha}}\leq B^{j,\alpha}$  $\forall 1 \leq j \leq k$

\item $f$ is $G$-equivariant
\end{enumerate}
\end{definition}

When there is no confusion we will simply refer to such functions as smooth GH approximations.

Our main purpose is to show the following

\begin{theorem}\label{thm_smooth_GH}
Let $(M_{0}^{n},g_{0})$ and $(M_{1}^{m},g_{1})$ be Riemannian $G$-manifolds which are $\{A\}^{k}_{0}$-regular, possibly with boundary, with $inj^{B}(M_{1})\geq \iota >0$, $\iota\leq \sqrt{A^{0}}$ and $n\geq m$.  Then for each $0\leq\alpha<1$ there exists $\{B\}^{k+2,\alpha}_{0}=\{B(n,A,\iota,\alpha,G)\}^{k+2,\alpha}_{0}$ so that $\forall \epsilon>0$ $\exists$ $\delta=\delta(n,A,\iota,\alpha,G,\epsilon)>0$ such that if $f:M_{0,\iota/2}\rightarrow M_{1,\iota/2}$ is an $(\sqrt{A^{0}},\delta)$ $G$-equivariant GH approximation, then there exists a $\{B\}^{k+2,\alpha}_{0}$-bounded $(\sqrt{A^{0}},\epsilon)$ $G$-eGH approximation $f_{\epsilon}:M_{0,\iota}\rightarrow M_{1}$, with $M_{1,2\iota}\subseteq f_{\epsilon}(M_{0,\iota}) \subseteq M_{1,\iota/2}$.
\end{theorem}
\begin{remark}
The assumption $f:M_{0,\iota/2}\rightarrow M_{1,\iota/2}$, as opposed to $f:M_{0}\rightarrow M_{1}$, simply guarantees that there do no exist any small holes in $M_{0}$ that do not correspond to holes in $M_{1}$.
\end{remark}

Before going to the proof we make a brief remark about the center of mass technique, as in \cite{BK}, which is used in the paper.  Namely if $(M^{n}_{0,g_{0}})$ and $(M^{m}_{1},g_{1})$ are smooth Riemannian manifolds with $|sec_{g_{i}}|\leq 1$ and $inj(M_{1})\geq \iota > 0$, then we know that if $f:M_{0}\rightarrow M_{1}$ is measurable and $\epsilon>0$ is such that $f(B_{\epsilon}(x))\subseteq B_{\iota/4}(f(x))$ for each $x\in M_{0}$ then we can apply the center of mass technique to $f$ to get a nearby map $f_{\epsilon}$ which is smooth.  However with this technique we \textit{apriori} have only $C^{2}$ bounds, which depend on $n,m,\epsilon$, for the map $f_{\epsilon}$ because of our use of the distance functions on $M_{0}$ and $M_{1}$.  This is not quite sufficient since curvature bounds on $M_{i}$ should imply an optimal bound of $C^{2,\alpha}$ for our functions $f_{\epsilon}$.

Now there are several ways of dealing with this problem, for instance using harmonic approximations for the distance functions during the averaging process.  The approach we will take is slightly more convenient for other constructions and is as follows:  Fix a small number $\xi << 1$ once and for all for this construction.  As in \cite{A} and \cite{PWY} we can find metrics $\tilde g_{0}$ on $M_{0,\iota}$ and $\tilde g_{1}$ on $M_{1,\iota}$ which are $\{A\}^{\infty}_{0}$-regular such that $g_{i}e^{-\xi}\leq \tilde g_{i} \leq g_{i}e^{\xi}$ and so that in appropriate weak coordinates around each point the new and old metrics are within an \textit{apriori} $C^{1,\alpha}$ distance from one another (depending only on $n,\iota,\alpha$).  In particular the injectivity radius of $M_{1,\iota}$ is at least $\iota/2$.  Now we may do the center of mass technique on $f$ with respect to our new metrics to get the function $\tilde f_{\epsilon}$. Because the perturbed metrics are $\{A\}^{\infty}_{0}$-regular it is easy to check that $\tilde f_{\epsilon}$ has \textit{apriori} $C^{\infty}$ bounds with respect to the $\tilde g_{i}$ geometries.  Since the $g_{i}$ and $\tilde g_{i}$ have bounded $C^{1,\alpha}$ distance this thus tells us that $\tilde f_{\epsilon}$ has \textit{apriori} $C^{2,\alpha}$ bounds with respect to the $g_{i}$ geometries, which only depend on $n,m,\epsilon,\alpha$.  Of course more generally had we assumed the $g_{i}$ were $\{A\}^{k}_{0}$-regular then the same trick tells us that $\tilde f_{\epsilon}$ is $\{B\}^{k+2,\alpha}_{0}=\{B(n,m,A,\epsilon,\alpha)\}^{k+2,\alpha}_{0}$-regular.

Now we can continue with the proof:

\begin{proof}
By scaling we can assume $A^{0}=1$ without loss of generality (so $\iota\leq 1$), and we can assume $\delta < \iota/100$.  As in \cite{BK} if $\delta$ is sufficiently small, depending only on $n,A^{0}$ and $\iota$, we can assume $f$ is $G$ equivariant by a center of mass averaging.  So there is no loss in assuming this.  Now let us pick a $\lambda = \mu\iota$ with $\mu<< 1$, the number $\mu$ will be chosen later but will be considered fixed and will depend only on $n$ and $A^{0}$.  Let $\tilde g_{0}$ and $\tilde g_{1}$ be $\{C\}^{\infty}_{0}=\{C(n,A,\iota)\}^{\infty}_{0}$-regular metrics on $M_{0,\iota/2}$,$M_{1,\iota/2}$ as in \cite{A} \cite{PWY} for which in appropriate weak coordinate systems with $k+2+\alpha$ bounds the metrics $\tilde g_{i}$ and $g_{i}$ are \textit{apriori} $k+1+\alpha$ close.  As in the paragraphs proceeding this proof by the perturbed $\lambda$-center of mass we can construct $\tilde f_{\lambda}:M_{0,3\iota/4}\rightarrow M_{1,3\iota/4}$.  Notice that $f_{\lambda}$ has $\{B\}^{\infty}_{0}=\{B(n,A,\lambda)\}^{\infty}_{0}$ bounds with respect to the metrics $\tilde g_{0}$ and $\tilde g_{1}$, and hence $\{B\}^{k+2,\alpha}_{0}=\{B(n,A,\lambda)\}^{k+2,\alpha}_{0}$ bounds with respect to the metrics $g_{0}$ and $g_{1}$.  Now the first observation is that by the proof of Theorem 2.6 of \cite{CFG} we can pick $\mu=\mu(n,A^{0})$ so that $\tilde f_{\lambda}$ is a $\frac{1}{2}$-Riemannian submersion with $diam(\tilde f_{\lambda}^{-1}(y))\leq c(n,A^{0})\delta$ for each $y\in \tilde f(M_{0,3\iota/4})$.

Now the claim is that for any $\epsilon>0$ we can pick $f_{\epsilon}\equiv \tilde f_{\lambda}$ as our desired map, for $\delta$ sufficiently small.  Notice in \cite{CFG} the corresponding $f_{\epsilon}$ is chosen by letting $\lambda\rightarrow 0$ as $\epsilon\rightarrow 0$.  This has the advantage of keeping the global GH approximation property (as opposed to the more local $(r,\epsilon)$ GH property which we will show is preserved), but makes higher order estimates less clear.  In fact one can carry out a similar proof to the below for the $\lambda\rightarrow 0$ case, but it is more technical and requires additional estimates.  However the intuition for why one should not need to make $\lambda\rightarrow 0$ is that if we write $\tilde f_{\lambda}$ in local weak coordinates then we see $\tilde f_{\lambda}$ will have a very dense discrete symmetry forced upon it.  The derivative bounds on $\tilde f_{\lambda}$ thus force this to correspond to an 'almost' smooth symmetry at every point.  In particular as $\delta\rightarrow 0$ it is not only a $\frac{1}{2}$-Riemannian submersion, but becoming a $1$-Riemannian submersion and thus the local GH property is preserved.  We make this precise as follows:

Assume the claim is false for some $\epsilon>0$.  Then we can find $\delta_{i}\rightarrow 0$ and $(M_{0,i},g_{0,i})$, $(M_{1,i},g_{1,i})$ which are $\{A\}^{k}_{0}$-regular with $f_{i}:M_{0,i,\iota/2}\rightarrow M_{1,i,\iota/2}$ $(1,\delta_{i})$ $G$-equivariant GH approximations, but such that $f_{i,\epsilon}|_{M_{0,i,\iota}}$ are not $\{B\}^{k+2,\alpha}_{0}$-smooth $G$-$(1,\epsilon)$-GH approximations.  Now the smoothness conditions are certainly satisfied for the $f_{i,\epsilon}$, and if it held at each point $x\in M_{0,\iota}$ that $1-\epsilon \leq df[X]\leq 1+\epsilon$ then one could check that the $f_{i,\epsilon}$ are $(1,\epsilon)$ GH approximations.  So there must be points $x_{i}\in M_{0,i,\iota}$ such that the $f_{i,\epsilon}$ are not $\epsilon$-Riemannian submersions in either the $C^{1}$ or in the $C^{2}$ sense.  Now let $r=\iota/16$, $y_{i}=f(x_{i})$, $Y_{i}=B_{2r}(y_{i})$ and $X_{i}=f^{-1}_{i,\epsilon}(Y_{i})$.  Let $\tilde X_{i}$ be the universal cover of $X_{i}$ with $\tilde x_{i}$ a lift of $x_{i}$, $\Lambda_{i}$ its fundamental group and $\tilde f_{i,\epsilon}:\tilde X_{i}\rightarrow Y_{i}$ the lift of $f_{i,\epsilon}$.  Now there is an injectivity radius bound in $Y_{i}$, and hence as \cite{CFG},\cite{NaTi} an injectivity radius bound in $\tilde X_{i}$.  By the usual compactness we can pass to a subsequence so that $(\tilde X_{i},g_{i},\Lambda_{i},\tilde x_{i})\stackrel{C^{k+1,\alpha}-eGH}{\rightarrow}(\tilde X_{\infty},g_{\infty},N,\tilde x_{\infty})$ and $(\tilde Y_{i},h_{i},y_{i}) \stackrel{C^{k+1,\alpha}}{\rightarrow} (\tilde Y_{\infty},h_{\infty},y_{\infty})$, where as in \cite{NaTi} $N$ is a finite extension of a nilpotent, $(\tilde X_{\infty},g_{\infty},\tilde x_{\infty})$ and $(\tilde Y_{\infty},h_{\infty},y_{\infty})$ have $(\{C\}^{k+1,\alpha}_{0},r_{h})$-harmonic bounded geometry at $\tilde x_{\infty}, y_{\infty}$, and $\tilde X_{\infty}/N\approx Y_{\infty}$.  Recalling that the $\tilde f_{i,\epsilon}$ have $C^{\infty}$ bounds with respect to a nearby metric (and \textit{apriori} $C^{k+2,\sigma}$ bounds $\forall \sigma<1$ with respect to the original metrics) we can pass to a subsequence to limit $\tilde f_{i,\epsilon}\rightarrow \tilde f_{\infty,\epsilon}:\tilde X_{\infty}\rightarrow Y_{\infty}$.  We know the diameter of the fibers of $f_{i,\epsilon}$ are on the order of $\delta_{i}$ and so tending to zero.  The fibers of $\tilde f_{i,\epsilon}$ are invariant under the $\Lambda_{i}$ action and by the last statement the orbit of any point by $\Lambda_{i}$ is on the order of $\delta_{i}$ dense in the corresponding fiber of $f_{i,\epsilon}$.  Hence the fibers of the limit function $\tilde f_{\infty,\epsilon}$ are equal to the $N$ orbits in $\tilde X_{\infty}$ and so $\tilde f_{\infty,\epsilon}$ is equal to the quotient map $\tilde X_{\infty}\stackrel{N}{\rightarrow} Y_{\infty}$.  But then $\tilde f_{\infty,\epsilon}$ is a Riemannian submersion, and since the convergence was in at least $C^{2}$ that means for sufficiently large $i$ $\tilde f_{i,\epsilon}$ is an $\epsilon$-Riemannian submersion at $\tilde x_{i}$, which of course is a contradiction and thus proves the theorem.
\end{proof}

From the above it is only a little more work to get the following smoothing lemma:

\begin{lemma}[Smoothing Lemma]\label{lem_smoothing}
Additionally in the last lemma, there are $\{C\}^{\infty}_{0}=\{C(n,A,\iota)\}^{\infty}_{0}$-regular $G$-equivariant metrics $\bar g_{0}$, $\bar g_{1}$ on $M_{0,\iota/2}$ and $M_{1,\iota/2}$, respectively, with $||g_{i}-\bar g_{i}||_{C^{k+1,\alpha}}\leq D=D(n,A,k,\alpha)$ such that $f_{\epsilon}:(M_{0,\iota/2},\bar g_{0})\rightarrow (M_{1,\iota/2},\bar g_{1})$ is a Riemannian submersion which is a $(\sqrt A^{0},\epsilon)$ GH approximation.
\end{lemma}
\begin{remark}
If the $M_{i}$ do not contain boundary then the constants $\{C\}^{\infty}_{0}$ can be taken to not depend on $\iota$.
\end{remark}
\begin{proof}
Recall in the last proof we began by perturbing $g_{0},g_{1}$ to metrics $\tilde g_{0},\tilde g_{1}$ with all the required properties except for $f_{\epsilon}$ being a Riemannian submersion and GH approximation.  However since $f_{\epsilon}$ is a submersion we can define new metrics by $\bar g_{1}=\tilde g_{1}$ and $\bar g_{0}$ by changing $\tilde g_{0}$ along the horizontal components to make it into a Riemannian submersion.  Since $f_{\epsilon}$ is an $\epsilon$-Riemannian submersion with $C^{\infty}$ bounds with respect to $\tilde g_{0},\tilde g_{1}$ we have that $\bar g_{0}$ and $\tilde g_{0}$ are also at an \textit{apriori} $C^{\infty}$ bounded distance.  We just need to check that $f_{\epsilon}$ now defines a $(\sqrt A^{0},\epsilon)$ GH approximation.  For this notice that with respect to the original metric $g_{0}$ we have $diam_{g_{0}}(f^{-1}(y))\leq c(n,A^{0})\delta$.  But $g_{0}$ is uniformly close to $\bar g_{0}$, and so after altering $c$ we still have $diam_{\bar g_{0}}(f^{-1}(y))\leq c(n,A^{0})\delta$.  This combined with $f_{\epsilon}$ being a Riemannian submersion proves the claim.
\end{proof}

This lemma will in effect allow us to replace collapsing manifolds with partial regularity by nearby collapsing manifolds with full regularity.  This is convenient for purely technical reasons and with it we can more or less immediately take the constructions from \cite{CFG}, which require more regularity assumptions than one would like, and easily get the corresponding results without the regularity constraints.  In particular (see Appendix \ref{sec_inf_bund} for definitions):

\begin{theorem}\label{thm_nilbundleexist}
Let $(M_{0}^{n},g_{0})$ and $(M_{1}^{m},g_{1})$ be Riemannian $G$-manifolds which are $\{A\}^{k}_{0}$-regular, possibly with boundary, with $inj^{B}(M_{1})\geq \iota >0$, $\iota\leq \sqrt{A^{0}}$ and $n\geq m$.  Then for every $\epsilon>0$ there exists $\delta>0$ such that if $f:M_{0,\iota/4}\rightarrow M_{1,\iota/4}$ is a $(\sqrt A^{0},\delta)$ $G$-equivariant GH approximation, then there exists over $M_{0,\iota/2}$ and $M_{1,\iota/2}$ a $\{B\}^{k+2,\alpha}_{0} =\{B(n,A,\iota,\alpha)\}^{k+2,\alpha}_{0}$-regular $(\sqrt A^{0},\epsilon)$-$G$-equivariant infranil bundle structure $(f_{\epsilon},\nabla^{f_{\epsilon}})$.
\end{theorem}
\begin{proof}
Let $f_{\epsilon}$, $\bar g_{0}$ and $\bar g_{1}$ as guaranteed by the last lemmas.  Then the existence of $\{C\}^{\infty}_{0}$-regular (wrt $\bar g_{0}$ and $\bar g_{1}$) $\nabla^{f_{\epsilon}}$ follows immediately from the construction in \cite{CFG}.  Because $\bar g_{i}$ are $C^{k+1,\alpha}$ bounded distance from $g_{i}$ this gives the result.
\end{proof}

From this we get one more technical tool:

\begin{theorem}\label{thm_reg_inf_atl}
As in the last theorem if we let $r\leq r(n,A^{0},\iota)$ and $\{U_{\beta}\} = \{B_{r}(y_{\beta})\}\subseteq M_{1,\iota/2}$ be a covering of $M_{1,\iota}$, then there exists a simply connected nilpotent $N$ and a $\{B\}^{k+2,\alpha}_{0} =\{B(n,A,r,\alpha)\}^{k+2,\alpha}_{0}$-regular $G$-infranil atlas $\{U_{\beta}\times N,\Lambda,\varphi_{\beta}\}$ which is compatible with $(f_{\epsilon},\nabla^{f_{\epsilon}})$.
\end{theorem}
\begin{proof}
As in Theorem \ref{thm_infra_atlas2} we have that with respect to the metrics $\bar g_{0}$ and $\bar g_{1}$ there is such an atlas with $C^{\infty}$ bounds.  As usual because $\bar g_{i}$ are $C^{k+1,\alpha}$ bounded distance from $g_{i}$ this gives the result.
\end{proof}

\section{Construction of Frame Space}\label{sec_framespace}

In this section consider a sequence of collapsing Riemannian manifolds $(M^{n}_{i},g_{i})\stackrel{GH}{\rightarrow} (X,d)$ where the $(M_{i},g_{i})$ are $\{A\}^{k}_{0}$-regular.  We wish to construct what will be the frame space $Y_{\iota}\rightarrow X_{\iota}$ in Theorem \ref{thm_nt2_main1}.  With the exception of a regularity property the first part of the construction goes back to \cite{F1}.  The second part will use the results of the previous sections.  To begin with consider the $O(n)$ frame bundles $FM_{i}$ above each $M_{i}$.  If we fix a standard bi-invariant metric on $O(n)$ and the Levi-Civita connection $w$ on $FM_{i}$ to define a horizontal distribution then there is a unique $O(n)$-invariant metric $g^{FM,w}_{i}$ on $FM_{i}$ such that the quotient map becomes a Riemannian submersion, $w$ defines the perpendicular to the vertical in $FM_{i}$ and the restriction to each fiber is isometric $O(n)$.  Unfortunately the metrics $g^{FM,w}_{i}$ can have worse regularity properties than the metrics $g_{i}$ down below.  A simple computation (see \cite{B}) shows that this is because the curvature upstairs on $g^{FM,w}_{i}$ depends not only on the curvature down below and on $O(n)$, but depends on both the curvature of $w$ and on the covariant derivative of the curvature of $w$.  If $w$ is the Levi-Civita connection the end effect is that the curvature upstairs depends on the curvature of $g_{i}$ and on the covariant derivative of the curvature of $g_{i}$.  To remedy this we perturb our choice of connection slightly.

\begin{lemma}\label{lem_frame}
Let $(M^{n},g)$ be a $\{A\}^{k}_{0}$-regular Riemannian manifold, $\iota>0$ and $FM_{\iota}$ be the frame bundle above $M_{\iota}$.  Then there exists a metric $g^{FM}$ on $FM_{\iota}$ such that $(FM_{\iota},g^{F})$ is a $\{B\}^{k}_{0}=\{B(n,A,\iota)\}^{k}_{0}$-regular space, each fiber is isometric to the standard $O(n)$ and $(FM_{\iota}, g^{FM})/O(n)$ is isometric to $(M_{\iota},g)$.
\end{lemma}
\begin{proof}
The construction is as in the last paragraph with a minor perturbation.  We first smooth $g$ on $M_{\iota}$ as in \cite{A},\cite{PWY} on the balls of size $\iota/2$ to a $\{C(n,A,\iota)\}^{\infty}_{0}$-regular metric $g_{\iota}$.  If we construct $g^{FM,w}_{\iota}$ on $FM_{\iota}$ as in the last paragraph we thus get a $\{D(n,A,\iota)\}^{\infty}_{0}$-regular metric with all the desired properties except of course that $(FM_{\iota}, g_{\iota}^{FM,w})/O(n)$ is isometric to $(M_{\iota},g_{\iota})$.  This is easily fixed by modifying $g_{\iota}^{FM,w}$ by pulling back $g$ in the horizontal directions to construct a new metric $g^{FM}$.  Now because $g_{\iota}^{FM,w}$ has full regularity the standard submersion equations \cite{B} easily check for us that $g^{FM}$ is $\{B\}^{k}_{0}=\{B(n,A,\iota)\}^{k}_{0}$-regular as claimed.
\end{proof}

So far we have only used a metric on a fixed manifold $M$ to construct a corresponding metric on $FM$.  Now we wish to expand a little and given a sequence $(M^{n}_{i},g_{i})$ of $\{A\}^{k}_{0}$-regular spaces associate to each a $(\{B\}^{k+1,\alpha}_{0},r) =(\{B(n,A,\iota ,\alpha)\}^{k+1,\alpha}_{0}, r(n,\iota))$ -bounded metric $\bar g^{FM}_{i}$ on $FM_{i,\iota}$ that is in an appropriate sense compatible with the collapsing sequence.  To begin with let $g^{FM}_{i}$ be the metric from the last lemma on the slightly bigger space $FM_{i,\iota/4}$.  After possibly passing to a subsequence we get that $(FM_{i,\iota/4},g^{FM}_{i},O(n)) \stackrel{eGH}{\rightarrow} (Y_{\iota/4}, g^{Y},O(n))$ and as in \cite{F1} we see that the Gromov Hausdorff limit $(Y_{\iota/4}, g^{Y})$ is a Riemannian manifold.  Further it was shown in \cite{F1} that if $(FM_{i,\iota/4},g^{FM}_{i},O(n))$ were $\{B\}^{\infty}_{0}$-regular then $(Y_{\iota/2}, g^{Y},O(n))$ would be $\{C\}^{\infty}_{0}=\{C(n,B,\iota)\}^{\infty}_{0}$-regular.  We are dealing with lesser regularity here, but the verbatim contradiction proof as in \cite{F1} shows that if $(FM_{i,\iota/4},g^{FM}_{i},O(n))$ are $\{B\}^{k}_{0}$-regular then $(Y_{3\iota/8}, g^{Y})$ is $(\{C\}^{k+1,\alpha}_{0},r) =(\{C(n,B,\iota,\alpha)\}^{k+1,\alpha}_{0},r(n,B^{0},\iota))$-bounded for each $0<\alpha<1$.

The next step is to perturb the metrics $g^{FM}_{i}$ slightly.  Because $(FM_{i,3\iota/8}, g^{FM}_{i},O(n)) \stackrel{eGH}{\rightarrow} (Y_{3\iota/8}, g^{Y},O(n))$ after passing to another subsequence we can use Theorem \ref{thm_smooth_GH} to find $\{D\}^{k+2,\alpha}_{1} = \{D(n,A,\iota,\alpha)\}^{k+2,\alpha}_{1}$-bounded $(r,\epsilon_{i})$-$O(n)$-Gromov Hausdorff Approximations $f_{i}:FM_{\iota/2,i}\rightarrow Y_{\iota/2}$ with $\epsilon_{i}\rightarrow 0$.  Combining with Theorem \ref{thm_nilbundleexist} we can construct a $\{D\}^{k+2,\alpha}_{0}$-regular (modify $D$'s possibly, but depends on same variables) $O(n)$ nil bundle $(f_{i},\nabla^{f_{i}})$.  Finally we can use Theorem \ref{thm_reg_inf_atl} to find a covering $U_{\beta}$ of $Y_{\iota}$ so that we have compatible $\{D\}^{k+2,\alpha}_{0}$-regular nil atlases $\{U_{\beta}\times N_{i},\Lambda_{i},\varphi_{\beta,i}\}$ of $FM_{i,\iota}$ over $Y_{\iota}$.  Now the $g^{FM}_{i}$ are not quite compatible with this nil atlas structure, which is to say they are not invariant under the local $N_{i}$ actions.  They are however invariant under the action of $\Lambda_{i}$, a cocompact subgroup whose orbits are $\epsilon_{i}$-dense.  Following \cite{CFG} we can thus average to get $(\{C\}^{k+1,\alpha}_{0},r) = (\{C(n,A,\iota,\alpha)\}^{k+1,\alpha}_{0},r(n,\iota))$-bounded metrics $\bar g^{FM}_{i}$ which are now compatible with the underlying nil atlas structures.  Notice that it may not be the case that $(FM_{i},\bar g^{FM}_{i})/O(n)$ is isometric to $(M_{i},g_{i})$, but because the perturbations are becoming increasingly small we do have $||g^{FM}_{i}-\bar g^{FM}_{i}||_{C^{k+1,\alpha}}\rightarrow 0$.

\section{Approximations of $V^{T},\rho,V^{ad}$}\label{sec_approx_V_str}

We are now nearly in a position to combine the results of the previous sections to prove Theorem \ref{thm_nt2_main1} and construct the bundles $V^{T}$ and $V^{ad}$.  The constructions of the last sections have produced equivariant nil atlases over the collapsing frame bundles, we need now to directly analyze the nil atlas structure itself.  Given our collapsing sequence $(FM_{i},\bar g^{FM}_{i},O(n))\stackrel{eGH}{\rightarrow}(Y,g^{Y},O(n))$ with the induced nil structures from the last section the goal will be to first construct a sequence of suitable equivariant vector bundles $V^{T}_{i}\rightarrow Y$ and $V^{ad}_{i}\rightarrow Y$.  These bundles will have properties very similar to our desired bundles $V^{T}$, $V^{ad}$ and in fact in the next section we will see that after passing to subsequences that the $V^{T}_{i}$ and $V^{ad}_{i}$ will limit in a suitably strong sense to our desired bundles $V^{T}$ and $V^{ad}$.

Because the constructions of this section are for each $FM_{i}$ individually and rely only on the underlying nil structures we will limit ourselves for now to considering a $\{A\}^{k+2,\alpha}_{0}$-bounded equivariant nil atlas $\{(U_{\beta}\times N,\Lambda,\varphi_{\beta},h^{coord}_{\beta})\}$ for the $\{A\}^{k+1,\alpha}_{0}$-bounded spaces $(M_{0},g_{0},G)$ over $(M_{1},g_{1},G)$.  We begin locally by associating for each $\beta$ two vector bundles over $U_{\beta}$, namely the bundle $U_{\beta}\times\eta$ of right invariant vertical vector fields in $U_{\beta}\times N$ and the bundle $T(U_{\beta}\times N)/N\approx TU_{\beta}\times\eta\rightarrow U_{\beta}$ of all right invariant vector fields in $U_{\beta}\times N$.  The identification $T(U_{\beta}\times N)/N\approx TU_{\beta}\times\eta$ is derived from the natural identification $T(U_{\beta}\times N)\approx TU_{\beta}\times TN$.  Then the map $\rho_{\beta}$ is just the projection map $\rho_{\beta}:TU_{\beta}\times\eta\rightarrow TU_{\beta}$.  Further the coordinate metrics $h^{coord}_{\beta}$ on $U_{\beta}\times N$ then descend to fiber metrics, also denoted as $h^{coord}_{\beta}$, on $TU_{\beta}\times\eta$ and $U_{\beta}\times\eta$.

To understand the global picture take two charts $U_{\beta}\cap U_{\gamma}\equiv U_{\beta\gamma}\neq\emptyset$.  Then we may consider the transition map $\varphi_{\beta\gamma}=\varphi_{\gamma}^{-1}\circ \varphi_{\beta}:U_{\beta\gamma}\times(N/\Lambda)\rightarrow U_{\beta\gamma}\times(N/\Lambda)$.  Because this map is an isometry with respect to the metric $g_{0}$ on $M_{0}$ and because the coordinates are $\{A\}^{k+2,\alpha}_{0}$-bounded we see that $\varphi_{\beta\gamma}$ is $\{A\}^{k+2,\alpha}_{0}$-bounded.  Further by construction its restriction to each nil fiber $N/\Lambda$ is an affine transformation and hence $\varphi_{\beta\gamma}$ maps right invariant vectors to right invariant vectors while preserving vertical vectors.  Therefore $\varphi_{\beta\gamma}$ induces well defined $\{A\}^{k+1,\alpha}_{0}$-bounded maps $\varphi^{T}_{\beta\gamma}:TU_{\beta\gamma}\times\eta \rightarrow TU_{\beta\gamma}\times\eta$ and $\varphi^{ad}_{\beta\gamma}:U_{\beta\gamma}\times\eta\rightarrow U_{\beta\gamma}\times\eta$.  Notice further that the induced mapping on the lie algebra $\eta$ at each point is by an element of $C_{Aff}\equiv \frac{Cent(N)}{Cent(N)\cap\Lambda}\rtimes Aut(\Lambda)$.  However $Cent(N)$ acts trivially on right invariant vector fields and so that the action is just by an element of $Aut(\Lambda)$.  Also we see that the maps $\rho_{\beta}$ commute with the coordinate transformations.  Thus the collections $(TU_{\beta}\times\eta,\varphi^{T}_{\beta\gamma})$ and $(U_{\beta}\times\eta,\varphi^{ad}_{\beta\gamma})$ construct $\{A\}^{k+1,\alpha}_{0}$-bounded vector bundles $V^{T}_{0},V^{ad}_{0}\rightarrow M_{1}$ with an induced vector bundle mapping $\rho_{0}:V^{T}\rightarrow TM_{1}$.  The structure group reduction of $V^{ad}_{0}$ to $Aut(\Lambda)$ constructs a canonical flat connection $\nabla^{flat}_{0}$ on $V^{ad}_{0}$.  Because the $G$ actions are by the central affine transformations $C_{Aff}$ on $N/\Lambda$ this makes them into an equivariant vector bundles with $\nabla^{flat}_{0}$ $G$-invariant.  Summarizing we get

\begin{theorem}\label{thm_vector_bundles}
Let $(M^{n}_{0},g_{0},G)$ and $(M_{1},g_{1},G)$ be $(\{A\}^{k+1,\alpha}_{0},r)$-regular and let $\{(U_{\beta}\times N,\Lambda,\varphi_{\beta},h^{coord}_{\beta})\}$ be a $\{B\}^{k+2,\alpha}_{0}$-bounded equivariant nil atlas with compatible nil bundle structure $(f,\nabla^{f})$.  Then there exist equivariant vector bundles $V^{T}_{0},V^{ad}_{0}\rightarrow M_{1}$ and an equivariant mapping $\rho_{0}:V^{T}_{0}\rightarrow TM_{1}$ with $\{C\}^{k+1,\alpha}_{0} =\{C(A,B,r,n)\}^{k+1,\alpha}_{0}$-bounded local trivializations $\{(TU_{\beta}\times\eta, \varphi_{\beta\gamma}^{T})\}$, $\{(U_{\beta}\times\eta, \varphi_{\beta\gamma}^{ad})\}$ such that
\begin{enumerate}
\item In the local trivializations $\rho_{0}:TU_{\beta}\times\eta\rightarrow TU_{\beta}$ is just the projection to the first factor.

\item If $ad_{0}:\eta\times\eta\rightarrow\eta$ is the pointwise adjoint map then $||ad||_{h^{coord}_{\beta}}\leq 1$ in each trivialization.

\item The transition maps $\varphi_{\beta\gamma}^{T}$ and $\varphi_{\beta\gamma}^{ad}$ act as lie algebra automorphisms on the $\eta$ factors which are independent of $x\in U_{\beta}\cap U_{\gamma}$.

\item There exists an equivariant diffeomorphism $\varphi^{T}_{0}:TM_{0}\rightarrow f^{*}V^{T}_{0}$.
\end{enumerate}
\end{theorem}

If the metric $g_{0}$ above is compatible with the nil atlas then we may put geometries on our associated bundles.  The invariance of $g_{0}$ in the weak coordinates $U_{\beta}\times N$ by the right $N$ action constructs on $TU_{\beta}\times\eta$ and $U_{\beta}\times\eta$ local fiber metrics. The transition maps $\varphi_{\beta\gamma}:U_{\beta\gamma}\times(N/\Lambda) \rightarrow U_{\beta\gamma}\times(N/\Lambda)$ are isometries with respect to $g_{0}$ and so these local fiber metrics induce global fiber metrics $g^{T}_{0}$,$h^{ad}_{0}$ on $V^{T}_{0}$ and $V^{ad}_{0}$ respectively.  In the case of $V^{T}_{0}$ the fiber metric $g^{T}_{0}$ completely recovers the original metric $g_{0}$ since the equivariant diffeomorphism $\varphi^{T}$ from Theorem \ref{thm_vector_bundles}.3 above is now an isometry of vector spaces.  For the bundle $V^{ad}_{0}\rightarrow M_{1}$ we get not only a fiber metric $h^{ad}_{0}$ but one additional piece of geometric information, namely that the invariant metric $g_{0}$ induces a horizontal distribution on $U_{\beta}\times N$ and hence an affine connection $\nabla^{ad}_{0}$ on $U_{\beta}\times\eta$ by the association of the $\eta$ factor with the local vertical invariant vector fields.  These geometries will be studied more in later sections.

\section{Construction of $V^{T},\rho,V^{ad}$}\label{sec_V_str}

The construction of the last section nearly proves for us Theorem \ref{thm_nt2_main1}.  The setup of this section is a sequence $(M^{n}_{i},g_{i})\rightarrow (X,d)$ of $\{A\}^{k}_{0}$-regular manifolds and the associated spaces $(FM_{i},\bar g^{FM}_{i},O(n))\rightarrow (Y,g^{Y},O(n))$ from Section \ref{sec_V_str}.  For each $i$ we may now apply the results of the last section to construct $O(n)$ equivariant vector bundles $V^{T}_{i},V^{ad}_{i}\rightarrow Y$ along with equivariant mappings $\rho_{i}:V^{T}_{i}\rightarrow TY$ with all the desired properties of Theorem \ref{thm_nt2_main1}.  To finish the theorem we show that after passing to a subsequence the vector bundles $\{V^{T}_{i}\}$ and $\{V^{ad}_{i}\}$ all become equivariantly vector bundle isomorphic.  Further we can pick $V^{T}$ with fiber metric $g^{T}$ so that the said isomorphisms $\phi_{i}:V^{T}\rightarrow V^{T}_{i}$ satisfy $||g^{T}-\phi_{i}^{*}g^{T}_{i}||_{C^{k+1,\alpha}}\rightarrow 0$.

This construction is in fact nearly immediate from Theorem \ref{thm_vector_bundles}.  Beginning with the $V^{T}$ bundle: if we take the trializations $\{(TU_{\beta}\times\eta_{i}, \varphi_{\beta\gamma,i}^{T})\}$ from Theorem \ref{thm_vector_bundles} then by construction the transition maps $\varphi_{\beta\gamma,i}^{T}$ are $\{C\}^{k+1,\alpha}_{0}$-bounded, as are the fiber metrics $g^{T}_{i}|_{U_{\beta}}$ for each $\beta$. We can thus pass to subsequences so that for each $\beta,\gamma$ the maps $\varphi_{\beta\gamma,i}^{T}$ and the metrics $g^{T}_{i}|_{U_{\beta}}$ converge in $C^{k+1,\alpha'}$ to maps $\varphi_{\beta\gamma}^{T}$ and fiber metrics $g^{T}_{\beta}$.  The limits satisfy cocycle conditions since the convergence is at least continuous and so define a global bundle $V^{T}$ with fiber metric $g^{T}$ as claimed.  Theorem \ref{thm_vector_bundles}.2 guarantees that $\eta_{i}\rightarrow\eta$, a limit nilpotent lie algebra, and the convergence of $\varphi^{T}_{\beta\gamma,i}$ tells us that $\varphi^{T}_{\beta\gamma}$ act by affine transformations on $\eta$ which are independent of $x\in U_{\beta}\cap U_{\gamma}$.

In the case of the $V^{ad}$ bundle the procedure is verbatim, however in addition to a limit fiber metric $h^{ad}$ we may limit out the induced affine connections $\nabla^{ad}_{i}$ and the canonical flat connection $\nabla^{flat}_{i}$ from of Section \ref{sec_approx_V_str} to get $G$-invariant affine connections $\nabla^{ad}$ and $\nabla^{flat}$ on $V^{ad}$.

\section{The Canonical Flat Geometry and the Limit Central Decomposition of $V^{ad}$}\label{sec_flat_conn}

The ability to use the bundles $V^{ad}$ and $V^{T}$ to do analysis on the limit space $X$ will be exploited in future sections but first we want to discuss how their internal structure captures properties of the collapse.  This is especially relevant when the $M_{i}$ only have curvature bounded away from some controlled subsets $S_{i}\rightarrow S\subseteq X$.  We have seen in the last sections how to construct a canonical flat connection $\nabla^{flat}$ on $V^{ad}$.  It turns out that the holonomy of this flat connection around $S$ describes a twisting of the $M_{i}$ during the collapsing process which is not possible in the bounded curvature scenario.  In particular this gives a necessary obstructions to removable singularity type theorems.

We would like to describe a refinement of this connection which yields more information about the bundle $V^{ad}$.  Namely we will construct a canonical sequence of nested subbundles $0\subseteq V^{c^{0}}\subseteq\ldots\subseteq V^{c^{k}} = V^{ad}$, each of which is invariant with respect to the flat connection on $V^{ad}$.  This nesting will be called the limit central decomposition of $V^{ad}$.  This nesting is useful even in the case when the collapsing manifolds have uniformly bounded curvature, and we will use it in Section \ref{sec_thm2proof} in the proof of Theorem \ref{thm_nt2_main2}.

The basis for the limit central decomposition is that the lie algebras $\eta_{i}$ and $\eta$ of the bundles $V^{ad}_{i}$ and $V^{ad}$, respectively, need not be the same.  The existence of the subbundles $V^{c^{a}}$ gives direct information about the changing of these lie algebras in the limit.  To understand the limit central decomposition of $V^{ad}$ we are going to introduce the central decomposition of the bundles $V^{ad}_{i}$ and their lie algebras $\eta_{i}$.

First a discussion of nilpotent lie algebras.  For an arbitrary nilpotent lie algebra $\eta=\eta^{0}$ let $c^{0}\equiv cent^{0}\leq \eta^{0}$ be its center.  Note that the center of a nilpotent lie algebra has positive dimension and so this is a nontrivial subspace.  Let $\eta^{1}\equiv \eta^{0}/cent^{0}$ be the quotient nilpotent lie algebra.  Then we have a vector space isomorphism $\eta^{0}\approx cent^{0}\oplus\eta^{1}$.  The space $cent^{0}$ clearly lives canonically as a subspace of $\eta^{0}$, however an embedding of $\eta^{1}$ into $\eta^{0}$ is the same as a choice of perpendicular $(cent^{0})^{\perp}$ and is not canonical.  If we repeat the process we can write $\eta^{1}\approx cent^{1}\oplus \eta^{2}$, where $cent^{1}$ is the center of $\eta^{1}$ and $\eta^{2}=\eta^{1}/cent^{1}$ is the quotient lie algebra.  Thus we have a vector space isomorphism $\eta^{0}\approx cent^{0}\oplus cent^{1}\oplus \eta^{2}$.  Again the spaces $\eta^{1},\eta^{2}$ do not canonically embed into $\eta^{0}$, however the subspace $cent^{0}\oplus cent^{1}\equiv c^{1}\subseteq \eta^{0}$ does.  This follows because $cent^{1}$ embeds canonically into $\eta^{1}$ and any two embeddings of $\eta^{1}$ into $\eta^{0}$ differ only by elements of $cent^{0}$.  Note that $c^{1}$ is even a subalgebra of $\eta^{0}$.  Now we can continue this process with $c^{a}\equiv cent^{0}\oplus\ldots\oplus cent^{a}\subseteq \eta^{0}$ and because each center is nontrivial this process must eventually terminate.  Hence we have a natural nesting $0\subseteq c^{0}\subseteq\ldots\subseteq c^{k}=\eta$.

\begin{definition}\label{def_nt2_centraldecom}
We call the nesting $0\subseteq c^{0}\subseteq\ldots\subseteq c^{k}=\eta$ of subalgebras $c^{a}$ the central decomposition of $\eta$.
\end{definition}

That each $c^{a}$ is in fact a subalgebra of $\eta$ can be proved inductively and hence an automorphism of $\eta$ preserves each $c^{a}$ and induces an automorphism on $c^{a}$.

Now we wish to apply the above decompositions of the lie algebras $\eta_{i}$ on a more global scale to our bundles $V^{ad}_{i}$. So first let $\{(U_{\beta}\times N_{i},\Lambda_{i},\varphi_{\beta,i},h^{coord}_{\beta,i})\}$ be the nil atlas structures of $FM_{i}$ over $Y$ as in the previous sections.  So the local trivializations of $V^{ad}_{i}$ are of the form $U_{\beta}\times \eta_{i}$.  Because the coordinate transformations are lie algebra automorphisms we see that the central decomposition nesting $0\subseteq c_{i}^{0}\subseteq\ldots\subseteq c_{i}^{k}=\eta_{i}$ is preserved and induces the nesting of vector bundles $0\subseteq V^{c^{0}}_{i}\subseteq\ldots\subseteq V^{c^{k}}_{i}=V^{ad}_{i}$ (strictly the $k$ should depend on $i$ here, but after passing to a subsequence we can assume otherwise).  Notice that these subbundles are also lie algebra bundles and that the flat connection $\nabla^{flat}_{i}$ on $V^{ad}_{i}$ restricts to a flat connection on each of the subbundles.  By letting $i$ tend to infinity we have naturally constructed a nesting $0\subseteq V^{c^{0}}\subseteq\ldots\subseteq V^{c^{k}}=V^{ad}$ of the bundle $V^{ad}$.

\begin{definition}\label{def_nt2_limitcentdecom}
We call the nesting $0\subseteq V^{c^{0}}\subseteq\ldots\subseteq V^{c^{k}}=V^{ad}$ the limit central decomposition of $V^{ad}$.
\end{definition}

Notice that the limit central decomposition of $V^{ad}$ is \textit{not} the central decomposition of $V^{ad}$.  In particular while it is certainly true that each fiber of $V^{c^{0}}$ lies in the center of the fibers of $V^{ad}$, it may be the case that the center of $V^{ad}$ is strictly larger.  This corresponds precisely with the fact that the lie algebra of $V^{ad}$ can be distinct from the lie algebras of the $V^{ad}_{i}$.

We would like a quick application of the limit central decomposition which will be used in Section \ref{sec_thm2proof}.  We have that for large $i$ the $FM_{i}$ fiber over $Y$ and we have our standard local trivializations $U_{\beta}\times (N_{i}/\Lambda_{i})$ of this bundle.  Using the fiberwise lie algebra exponential map on $U_{\beta}\times \eta_{i}$ we can pull back the lattice $\Lambda_{i}$ to a subset $exp^{-1}(\Lambda_{i})$ of $\eta_{i}$.  As in \cite{} we see that because $\eta_{i}$ is nilpotent that the integral span $L_{\beta,i}$ of $exp^{-1}(\Lambda_{i})$ is a vector space lattice.  In fact because the structure group of the trivializations $U_{\beta}\times \eta_{i}$ has been reduced to $Aut(\Lambda_{i})$ we see these lattices are globally well defined and give a lattice section $L_{i}\subseteq V^{ad}_{i}$.  More than that, because the $O(n)$ action on $V^{ad}_{i}$ also acts by $Aut(\Lambda_{i})$ in our local coordinate representations we have that the $L_{i}$ are $O(n)$ invariant.  Notice that because $\eta_{i}$ is nilpotent that the intersection of $L_{i}$ and $V^{c^{a}}_{i}$ forms a lattice for each element of the central decomposition.  We would like to exploit these points to give one more reduction of the structure group of the limit central decomposition.

To do that consider the following.  For the sake of good coordinates we have up to this point fixed inner products $h^{coord}_{\beta,i}$ on $\eta_{i}$ for each $\beta$.  By letting the inner products vary with $\beta$ we have guaranteed that the induced fiber metric $h^{ad}_{i}$ and connection $\nabla^{ad}_{i}$ from $FM_{i}$ had uniform bounds independent of $\beta$.  Instead now fix an inner product $h^{coord}_{i}$ on $\eta_{i}$.  Though $h^{ad}_{i},\nabla^{ad}_{i}$ are not bounded independent of $\beta$ with respect to $h^{coord}_{i}$, it is clear that with respect to some basepoint in $Y$ these bounds can be taken to degenerate at worst exponentially in the distance function.  Now fixing $h^{coord}_{i}$ has the following advantage.  Because the transformations preserve the lattice $L_{i}$, which is locally a constant lattice in each coordinate, with respect to a fixed background metric the transformation functions are now special linear.  More specifically if we fix an $h^{coord}_{i}$-orthonormal basis of $\eta_{i}$ such that the first $dim(c^{0})$ span $c^{0}$, the first $dim(c^{1})$ span $c^{1}$ and so forth then by writing each $U_{\beta}\times \eta_{i}$ in such coordinates we have reduced the structure group of each $V^{c^{a}}_{i}\subseteq V^{ad}_{i}$ to the special linear group.  This clearly limits to give a further reduction of the structure group of each $V^{c^{a}}\subseteq V^{ad}$ to the special linear group.

\section{The limit geometry of $V^{ad}$}\label{sec_geom_V_ad}

In order to use all this structure to do analysis we need to discuss two points.  To begin with we have so far built an equivariant bundles $V^{ad},V^{T}\rightarrow Y\stackrel{O(n)}{\rightarrow} X$ above $Y$ together with an equivariant fiber metrics $g^{T},h^{ad}$.  However these fiber metrics describe the collapsing geometry of the sequence $FM_{i}\rightarrow Y$.  While no doubt this information captures the collapsing sequence $M_{i}\rightarrow X$ it is important to be more explicit.  For this purpose we will introduce two more equivariant fiber metrics $g^{T,X}$ and $h^{ad,X}$ on the bundles $V^{T},V^{ad}$ respectively.  Once this is done we will compute the various differential equations satisfied by the fiber metrics and connections and see how they relate to the geometry of the spaces $M_{i}, FM_{i}, X$ and $Y$.

Now let us begin with a small open set $U^{X}=B_{\epsilon}(x)\subseteq X$.  By small here we mean $\epsilon$ is sufficiently small so that $U^{X}\approx \mathds{R}^{m}/\tilde T$ is a convex neighborhood of $x$.  Then we know we can pick neighborhoods $U_{i}=B_{\epsilon}(x_{i})\subseteq M_{i}$ with the universal covers $(\tilde U_{i},g_{i})\stackrel{C^{k,\alpha}}{\rightarrow}(\tilde U^{X},g_{\infty})$ converging such that there is an isometric $N=N^{0}\rtimes A^{N}$ action on $\tilde U^{X}$ with $(\tilde U^{X},g_{\infty})/N\approx (U^{X},d^{X})$, where $N^{0}$ is a connected nilpotent and $A^{N}$ is finite.  We let $(F\tilde U^{X},g^{FM}_{\infty})$ be the frame bundle above $\tilde U^{X}$ equipped with the natural metric (lemma \ref{lem_frame} say) and notice that $N$ lifts to a free action on $F\tilde U^{X}$ with $(F\tilde U^{X},g^{FM}_{\infty})/N\approx (U^{Y},g^{Y})\subseteq Y$ (to see this just notice that $(F\tilde U_{i},g^{FM}_{i})\rightarrow (F\tilde U^{X},g^{FM}_{\infty})$).

Now the invariant vectors (resp. vertical invariant vectors) on $F\tilde U^{X}$ correspond to sections of $V^{T}|_{U^{Y}}$ (resp. $V^{ad}|_{U^{Y}}$) and the induced inner product on these vectors gives us the fiber metric $g^{T}$ (resp. $h^{ad}$).  That is, if $V^{0},V^{1}\in \Gamma(V^{ad})$ and $\tilde V^{0},\tilde V^{1}\in \Gamma(TF\tilde U^{X})$ are the corresponding vertical vectors then $h^{ad}(V^{0},V^{1})(xN)\equiv g^{FM}_{\infty}(\tilde V^{0},\tilde V^{1})(x)$.  On the other hand if $\tilde V^{0},\tilde V^{1}$ are two invariant vectors on $F\tilde U^{X}$ then we can first project them into the perpendicular of the $O(n)$ orbits and then take their inner product to get a semidefinite fiber metric $g^{T,X}$ on $V^{T}$ (resp. $h^{ad,X}$ on $V^{ad}$).  That is $h^{ad,X}(V^{0},V^{1})\equiv g^{FM}_{\infty}(p_{O^{\perp}}(\tilde V^{0}),p_{O^{\perp}}(\tilde V^{1}))$.  Because the $N$ and $O(n)$ actions commute on $F\tilde U^{X}$ the projections of $\tilde V^{0},\tilde V^{1}$ are horizontal lifts of invariant vectors $\tilde V^{0,X},\tilde V^{1,X}$ on $\tilde U^{X}$ and the inner product satisfies $h^{ad,X}(V^{0},V^{1})=g^{FM}_{\infty}(p_{O^{\perp}}(\tilde V^{0}),p_{O^{\perp}}(\tilde V^{1})) = g_{\infty}(\tilde V^{0,X},\tilde V^{1,X})$.

It is instructive to see what this gives us when $U^{X}\subseteq X_{reg}$ is a subset of the regular part of $X$ and so is diffeomorphic to a ball in Euclidean space, and $N=N^{0}$ is connected (neither assumption is necessary for the following interpretation, it just makes notation more convenient).  In this case we have that $\tilde U^{X}\approx U^{X}\times N^{0}$ and $\tilde U^{Y}\approx U^{X}\times N^{0}\times O(n)$.  When we build the local adjoint bundles $U^{X}\times\eta\rightarrow U^{X}$ and $U^{X}\times O(n)\times\eta\rightarrow U^{X}\times O(n)$ the $N$ actions induce fiber metrics on these bundles.  The fiber metric $h^{ad}$ on $U^{X}\times O(n)\times\eta$ is of course the induced metric from this action while the fiber metric $h^{ad,X}$ constructed in the last paragraph on $U^{X}\times O(n)\times\eta$ is just the lift of the induced fiber metric on $U^{X}\times\eta$.  A similar statement holds for $g^{T,X}$ and so the fiber metrics $g^{T,X}$, $h^{ad,X}$ simply describe the unwrapped limit geometry above $X$.

Now to understand the equations satisfied by the geometric data we use the same interpretation as above.  Namely by using the unwrapped limits we view $(F\tilde U^{X},g^{FM}_{\infty})\stackrel{N}{\rightarrow} (U^{Y},g^{Y})$ as a principal bundle with $N$ acting isometrically.  We can locally identify $V^{ad}|_{U^{Y}}$ with the adjoint bundle of this principal bundle and the fiber metric and connection $h^{ad}|_{U^{Y}},\nabla^{ad}|_{U^{Y}}$ as the ones generated from this action.  Since the equations satisfied by the triple $(g^{Y},\nabla^{ad},h^{ad})$ are purely local computations these local unwrappings suffice to compute global equations.  When $U^{X}\subseteq X_{reg}$ then we may do the same to understand the geometry of $(g^{X},\nabla^{ad,X},h^{ad,X})$.

To do this the setup is fairly general and is as follows.  Let $P\stackrel{G}{\rightarrow}M$ be a principal $G$-bundle over a manifold $M$ and $\mathfrak{g}_{P}\rightarrow M$ be the associated adjoint bundle (we view $G$ as acting on the right to be consistent with the terminology in the rest of the paper).  If $h$ is a $G$ invariant metric on $P$ then it induces the associated triple $(\hat h,\hat\nabla,\check g)$, where $\check g$ is the quotient metric on $M$, $\hat h$ is the induced fiber metric on $\mathfrak{g}_{P}$ and $\hat\nabla$ is the induced affine connection on $\mathfrak{g}_{P}$.

We will need local coordinates in which to work and so we let $\{\check X^{ i }\}$ be a (local) vector basis on $M$ and $\{\hat V^{ a }\}$ on $\mathfrak{g}_{P}$.  It will be useful throughout to distinguish between horizontal and vertical entries so we will let $ a , b ,\ldots$ represent the vertical indices and $ i , j ,\ldots$ the horizontal indices.  Recall there is a one-to-one correspondence between the sections of $\mathfrak{g}_{P}$ and vertical right invariant vector fields on $P$ and so we let $\{V^{ a }\}$ be the vertical $G$-invariant vector fields associated to $\{\hat V^{ a }\}$.  We can also lift $\{\check X^{ i }\}$ horizontally with respect to $h$ to get a local horizontal basis $\{X^{ i }\}$ in $P$.  Then we point out that $\{X^{ i },V^{ a }\}$ forms a local basis on $P$ which are $G$-invariant.  Similarly given a $G$-invariant tensor $T$ on $P$ with values in $\mathcal{V}$ (resp. $\mathcal{H}$) we let $\hat T$ (resp. $\check T)$ be the associated section on tensor products of $\mathfrak{g}_{P}$, $TM$ and their dual's.

Our first order of business is to see how to compute covariant derivatives with respect to the induced connection $\hat\nabla$ on $\mathfrak{g}_{P}$.

\begin{lemma}\label{lem_adjoint_connection}
$\hat{\nabla}_{\check{X}^{ i }}\hat{V}^{ a } = \hat{[X^{i},V^{a}]}$.
\end{lemma}
\begin{proof}
Let $U\times G$ with $U\subseteq M$ be a local principal coordinate neighborhood and let $\hat{s}$ be a section of $\mathfrak{g}_{P}$ (hence a map $\hat{s}: U\times G \rightarrow \mathfrak{g}$ which is invariant under the $G$ action).  Then $\hat{s}(x,g)= Ad_{g}s_{x}= g^{-1}\cdot s_{x}\cdot g$ where $s_{x}\in \mathfrak{g}$.  Note that the induced right invariant vertical vector field on $U\times G$ is then $s(x,g)=s_{x}\cdot g$, and so $\hat{s}(x,g)=g^{-1}\cdot s(x,g)$.  Then we have that $\hat{\nabla}_{\check{X}^{ i }}\hat{s} = \partial_{X^{ i }}\hat{s}$. Let $\gamma(t)=(\gamma_{x}(t),\gamma_{g}(t))$ be a smooth curve with $\gamma(0)=(x,g)$ and $\dot{\gamma}(0)=X^{ i }$.  Then $\partial_{X^{ i }}\hat{s} = \frac{d}{dt}|_{t=0}((\gamma_{g}^{-1}(t))_{*}(s(x,g))) = g^{-1}_{*}([X_{g},s]) = \hat{[X_{g},s]}$, where $X_{g}$ is the $G$ component of $X$.  But we have that $\hat{[X_{g},s]} = \hat{([X,s]-[X_{x},s])} = \hat{[X,s]}$ because $X_{x}$ is tangent to $U$ and $s$ is tangent to $G$.
\end{proof}

Recall that if $A$ and $T$ are the O'Neill tensors on $P$ then we have that $T(V^{ a },V^{ b },X^{ i })= \langle\nabla_{V^{ a }}V^{ b },X^{ i }\rangle$ and $A(X^{ i },X^{ j },V^{ a }) =\langle\nabla_{X^{ i }}X^{ j }, V^{ a }\rangle$.  As with the notation above we can use these to define tensors $\hat T$ and $\check A$ on the bundles $\mathfrak{g}^{*}_{P}\otimes\mathfrak{g}^{*}_{P}\otimes T^{*}M$ and $T^{*}M\otimes T^{*}M\otimes\mathfrak{g}^{*}_{P}$ by the formulas $\hat T(\hat V^{ a },\hat V^{ b },\check X^{ i }) = T(V^{ a },V^{ b },X^{ i })$ and $\check A(\check X^{ i },\check X^{ j },\hat V^{ a }) = A(X^{ i },X^{ j },V^{ a })$.  As usual we may define the mean curvature vector $H$ as the trace of $T$ (this is following \cite{B} and may have a sign difference from other references).  Now we compute to see how our triple $(\hat h,\hat\nabla,\check g)$ relates to this information.

\begin{lemma}\label{lem_Veq1} We have that
\begin{enumerate}
\item $\hat{T}_{ a  b   i } = -\frac{1}{2}\hat{\nabla}_{ i }\hat{h}_{ a  b }$

\item $(\nabla_{ i }T_{ a  b   j })^{\wedge} = -\frac{1}{2}\hat{\nabla}^{2}_{ i   j }\hat{h}_{ a  b } + \frac{1}{4}\hat{h}^{\gamma\sigma} (\hat{\nabla}_{ i }\hat{h}_{ a \sigma}\hat{\nabla}_{ j }\hat{h}_{\gamma b } + \hat{\nabla}_{ i }\hat{h}_{\gamma b }\hat{\nabla}_{ j }\hat{h}_{ a \sigma})$
    \[
    = -\frac{1}{2}\hat{\nabla}^{2}_{ i   j }\hat{h}_{ a  b } + \hat{h}^{\gamma\sigma} (\hat{T}_{ a \sigma  i }\hat{T}_{\gamma b  \check by} + \hat{T}_{\gamma b   i }\hat{T}_{ a \sigma j})
    \]

\item $\hat{H}_{ i } \equiv \hat{h}^{ a  b }\hat{T}_{ a  b   i } = -\frac{1}{2}\hat{h}^{ a b }\hat{\nabla}_{ i }\hat{h}_{ a  b }$

\item $(g^{ i  j }\nabla_{ i }T_{ a  b   j })^{\wedge} = -\frac{1}{2}\hat{\triangle}\hat{h}_{ a  b } + \frac{1}{2}\hat{h}^{\gamma\sigma}\check{g}^{ i  j } (\hat{\nabla}_{ i }\hat{h}_{ a \gamma}\hat{\nabla}_{ j }\hat{h}_{\sigma b })$
    \[
    =  -\frac{1}{2}\hat{\triangle}\hat{h}_{ a  b } + 2\hat{h}^{\gamma\sigma}\check{g}^{ i j }T_{ a \gamma  i }T_{\sigma b   j }
    \]
\end{enumerate}
\end{lemma}
\begin{proof}
$1)$  We have
\[
T_{ a  b   i } = \langle\nabla_{V^{ a }}V^{ b },X^{ i }\rangle = -\langle V^{ b },\nabla_{V^{ a }}X^{ i }\rangle = -\langle V^{ b },\nabla_{X^{ i }}V^{ a }\rangle + \langle V^{ b },[X^{ i },V^{ a }]\rangle
\]
\[
= -\nabla_{X^{ i }}\langle V^{ a },V^{ b }\rangle + \langle\nabla_{X^{ i }}V^{ b },V^{ a }\rangle + \langle V^{ b },[X^{ i },V^{ a }]\rangle
\]
But $T_{ a  b   i } = T_{ b  a   i }$ so adding we get
\[
2T_{ a  b   i } = -2\nabla_{X^{ i }}\langle V^{ a },V^{ b }\rangle + \langle\nabla_{X^{ i }}V^{ a },V^{ b }\rangle + \langle\nabla_{X^{ i }}V^{ b },V^{ a }\rangle + \langle V^{ a },[X^{ i },V^{ b }]\rangle + \langle V^{ b },[X^{ i },V^{ a }]\rangle
\]
\[
= -\partial_{X^{ i }}\hat{h}_{ a  b } + \hat{\Gamma}^{\sigma}_{ i  a }\hat{h}_{\sigma b } + \hat{\Gamma}^{\sigma}_{ i  b }\hat{h}_{ a \sigma}
\]

$2)$ $\nabla_{ i }T_{ a  b   j } = \partial_{X^{ i }}T_{ a  b   j } - \Gamma^{\sigma}_{ i   a }T_{\sigma b   j } - \Gamma^{\sigma}_{ i   b }T_{ a \sigma  j } - \check{\Gamma}^{k}_{ i   j }T_{ a  b k}$

But we have that
\[
 \Gamma^{\sigma}_{ i   a } = \hat{h}^{\sigma\gamma}(\langle\nabla_{X^{ i }}V^{ a },V^{\gamma}\rangle) =
 \hat{h}^{\sigma\gamma}(\langle [X^{ i },V^{ a }],V^{\gamma}\rangle+\langle \nabla_{V^{a}}X^{i},V^{\gamma}\rangle)
\]

\[
= \hat{\Gamma}^{\sigma}_{ i   a } - \hat{h}^{\sigma\gamma}T_{ a \gamma  i }
\]

Plugging this into the first line yields the result.  The proof's of $(3)$ and $(4)$ are just by tracing.
\end{proof}

Finally we want to relate the curvatures, in particular the Ricci curvatures below though other curvatures are similar, of $P$ and $M$ to the differential geometry of the triple.  The following is direct from the O'Neill formulas \cite{B} and computations as in lemma \ref{lem_Veq1}.  Below we let $c_{abc}=\langle [V^{a},V^{b}],V^{c}\rangle$ be the structure coefficients of the lie algebra.

\begin{prop}\label{prop_curv_comp}
Let $\hat R_{ a  b }$ be the Ricci of the $G$-fiber, $\check R_{ i  j }$ be the Ricci on $M$, and let $R_{IJ}$ be the Ricci on $P$.  Then
\begin{enumerate}
\item $R_{ a  b } = \hat R_{ a  b } -\frac{1}{2}\hat\triangle\hat h_{ a  b } +\frac{1}{4}\hat h^{\gamma\sigma}\check g^{ i  j } (2\hat\nabla_{ i }\hat h_{ a \sigma}\hat\nabla_{ j }\hat h_{\gamma b } - \hat\nabla_{ i }\hat h_{\gamma\sigma}\hat\nabla_{ j }\hat h_{ a  b }) + \check g^{ i  j }\check g^{kl}\hat A_{i  k a }\hat A_{ j l b }$

\item $R_{a i} = \hat\nabla^{k}A_{kia} - H^{k}A_{kia} + h^{\sigma_{0}\sigma_{1}}h^{\gamma_{0}\gamma_{1}}(c_{\sigma_{0}a\gamma_{0}}T_{\sigma_{1}\gamma_{1}a} +c_{\gamma_{0}\sigma_{0}\sigma_{1}}T_{a\gamma_{1}i})$

\item $2\hat h^{ a  b }\check g^{kl}\check A_{ i  k a }\check A_{ j l b } + \frac{1}{4} \hat h^{ a  b }\hat h^{\gamma\sigma} \hat\nabla_{ i }\hat h_{ a \sigma}\hat\nabla_{ j }\hat h_{\gamma b } + R_{ i   j } = \check R_{ i   j } +\frac{1}{2}\mathcal{L}_{\check H}\check g_{ i   j }$
\end{enumerate}

or we can write it as

\begin{enumerate}
\item $R_{ a  b } = \hat{R}_{ a  b } -\frac{1}{2}\hat{\triangle}\hat{h}_{ a  b } - \check{H}^{ i }\hat{T}_{ a  b   i } + \hat{A}^{ i   j } _{\;\, a }\hat{A}_{ i   j  b } + 2\hat{T}_{ a }^{\;\sigma  i }\hat{T}_{ b \sigma  i }$

\item $R_{a i} = \hat\nabla^{k}A_{kia} - H^{k}A_{kia} + h^{\sigma_{0}\sigma_{1}}h^{\gamma_{0}\gamma_{1}}(c_{\sigma_{0}a\gamma_{0}}T_{\sigma_{1}\gamma_{1}a} +c_{\gamma_{0}\sigma_{0}\sigma_{1}}T_{a\gamma_{1}i})$

\item $\check{A}^{ i   j }_{\;\, a }\check{A}_{ i   j  b } + \check{T}^{ a  b }_{\;\;\;\; i }\check{T}_{ a  b   j } + R_{ i   j } = \check{R}_{ i   j } +\frac{1}{2}\mathcal{L}_{\check{H}}\check{g}_{ i   j }$
\end{enumerate}
\end{prop}

Notice that since we are on a principal bundle we may write $H=\nabla\mu$ as the gradient of the function $-ln\sqrt{det h}$.  With this we can view the first equation as a $\mu$-harmonic map for the fiber metric $h_{ a  b }$, the second equation as a Yang-Mills equations for the connection, and the third as a soliton type equation for the base metric.

\section{Proof of Theorems \ref{thm_nt2_main2} and \ref{cor_nt2_main1}}\label{sec_thm2proof}

The goal of this section is to prove Theorem \ref{thm_nt2_main2} and Theorem \ref{cor_nt2_main1}.

As was mentioned in the comments following the statement of Theorem \ref{thm_nt2_main2} the proof of this statement must have a global nature (the result fails if any of the hypothesis are removed) and relies on a series of maximum principals.  The proof is in several steps, which we quickly outline.  For the brief outline we assume the lie algebra $\eta$ of $V^{ad}\rightarrow Y\rightarrow X$ is abelian, getting around this assumption will involve the limit central decomposition constructed in Section \ref{sec_flat_conn}.  To begin with it was shown at the end of Section \ref{sec_flat_conn} that we can further reduce the structure group of $V^{ad}\rightarrow Y$ to the special linear group and in Section \ref{sec_geom_V_ad} we constructed the semidefinite fiber metric $h^{ad,X}$.  Recall that $h^{ad,X}$ is just the lift of the induced fiber metric from the nilpotent action on the local unwrappings above $X$.  Because of the structure group reduction the quantities $v^{ad,X}=\sqrt{det(h^{ad,X})}$ and $\mu^{ad,X}=-ln(v^{ad,X})$ are globally well defined.  If $(U^{X}\times N^{0},g^{U\times N})\rightarrow (U^{X},d)$ is a local unwrapped limit of the $M_{i}$ with $U^{X}\subseteq X_{reg}$, and hence $g^{U\times N}$ is Ricci flat with an isometric $N$ action, then it turns out that $\mu^{ad,X}$ is directly tied to the mean curvature (and indeed the integrability tensor) of this Riemannian submersion.  We will see that the $\mu^{ad,X}$ satisfies an interesting differential inequality (when $\eta$ is abelian at any rate) and that after a little work we will be able to apply a maximum principal to find that $\mu^{ad,X}=constant$.  From this we will be able to conclude that the mean curvature and integrability tensor of the unwrapped limits above $X$ actually vanish.  The final step is to show that the full second fundamental form of these unwrapped limits vanish, which involves another maximum type principal.  This will then tell us that $(U^{X}\times N^{0},g^{U\times N})$ is isometric to $U^{X}\times \mathds{R}^{k}$ with a product metric and in particular that if $g^{U\times N}$ is Ricci flat then so is $U$ with the quotient metric.

\begin{proof}[Proof of Theorem \ref{thm_nt2_main2}]

The beginning point of the actual proof comes from the limit central decomposition introduced in Section \ref{sec_flat_conn}.  Namely let $0\subseteq V^{c^{0}}\subseteq\ldots\subseteq V^{c^{k}}=V^{ad}$ be the limit central decomposition.  We have already discussed how the reduction of the structure group to the special linear group does in fact apply to each equivariant bundle $V^{c^{a}}$ and so for each $0\leq a\leq k$ we can define the quantities $v^{a,X}=\sqrt{det(h^{ad,X}|_{V^{c^{a}}})}$ and $\mu^{a,X}=-ln(v^{a,X})$.  We similarly define $v^{ad}=\sqrt{det(h^{ad})}$ and $\mu^{ad}=-ln(v^{ad})$ for the standard limit geometry on $V^{ad}$.  To see what type of equations the $\mu^{a,X}$ satisfy we need to interpret them.  For this let $U^{X}\subseteq X_{reg}$ be any small open set (being in the regular part is not needed, it just makes the following notation more convenient) with $(U^{X}\times N^{0},g^{U\times N})\stackrel{N}{\rightarrow} (U^{X},d^{X})$ the unwrapped limit geometry and $(U^{X}\times\eta,h^{U})$ the local adjoint bundle with the induced fiber metric $h^{U}$.  For $0\subseteq c^{0}\subseteq\ldots\subseteq c^{k}=\eta$ the limit central decomposition of $\eta$ we can then consider for each $a$ the induced bundle $(U^{X}\times c^{a},h^{a,U})$ where $h^{a,U}=h^{U}|_{c^{a}}$.  Then the function $\mu^{a,X}$ on $Y$ is just the lift of $\mu^{a,U}=-ln(\sqrt{ det(h^{a,U})})$ from $U^{X}$.  To understand this quantity note that for each $c^{a}$ we can consider the submersion $(U^{X}\times N^{0},g^{U\times N}) \stackrel{N}{\rightarrow} (U^{X},d^{X})$ in two steps, namely $(U^{X}\times N^{0},g^{U\times N}) \stackrel{C^{a}}{\rightarrow} (U^{X}\times (N/C^{a}),g^{U\times (N/C^{a})}) \stackrel{N/C^{a}}{\rightarrow} (U^{X},d^{X})$ by first quotienting out by the subgroup $C^{a}\leq N$ induced by $c^{a}\leq\eta$ and then by looking at the resulting Riemannian submersion over $U^{X}$.  We let $\mu^{a,U}$ define a function on $U^{X}\times (N^{0}/C^{a})$ by pulling it back by the submersion map and study first the Riemannian submersion $(U^{X}\times N^{0},g^{U\times N}) \stackrel{C^{a}}{\rightarrow} (U^{X}\times (N^{0}/C^{a}),g^{U\times (N/C^{a})})$.  If we let $H^{a}$ be the horizontal mean curvature vector field in $U^{X}\times N^{0}$ of this submersion then we note that $H^{a}$ is the horizontal lift of the vector field $\nabla\mu^{a,U}$ on $U^{X}\times (N^{0}/C^{a})$.  By tracing Proposition \ref{prop_curv_comp}.1 we get on $U^{X}\times (N^{0}/C^{a})$ that the $\mu^{a,U}$ satisfies the equation

\[
\triangle_{U^{X}\times (N/C^{a})}(\mu^{a,U})-\langle\nabla\mu^{a,U},\nabla\mu^{a,U}\rangle+|A^{C^{a}}|^{2}+R^{a}=0
\]

where $A^{C^{a}}$ is the integrability tensor of the Riemannian submersion $U^{X}\times N^{0}\rightarrow U^{X}\times (N^{0}/C^{a})$ and $R^{a}$ is the scalar curvature of the $C^{a}$ fiber above the corresponding point.  That the right hand side vanishes is because the metric $g^{U\times N}$ on $U^{X}\times N$ is Ricci flat.  Now $\mu^{a,U}$, $|A^{C^{a}}|$ and $R^{a}$ are all constant on each $N/C^{a}$ fiber and $U^{X}\times (N/C^{a})\rightarrow U^{X}$ is also a Riemannian submersion, so by replacing the vertical derivatives in $\triangle_{U^{X}\times (N/C^{a})}$ with second fundamental form terms of the $U^{X}\times(N^{0}/C^{a})\rightarrow U^{X}$ submersion we then see on $U^{X}$ that $\mu^{a,U}$ satisfies

\[
\triangle_{U^{X}}(\mu^{a,U})-\langle\nabla\mu^{a,U},\nabla\mu^{ad,X}\rangle+|A^{C^{a}}|^{2}+R^{a}=0
\]

We repeat one more time to find the equation satisfied by $\mu^{a,X}$ on $Y$.  That is, $\mu^{a,X}$ is the lift of $\mu^{a,U}$ and so now we need to add in vertical derivative terms to get the $Y$-laplacian, which tells us that as a function on $Y$ that the $\mu^{a,X}$ satisfies the equation

\[
\triangle_{Y}\mu^{a,X}-\langle\nabla\mu^{a,X},\nabla\mu^{ad}\rangle+|A^{C^{a}}|^{2}+R^{a}=0
\]

where of course $|A^{C^{a}}|^{2}+R^{a}$ is understood to be the corresponding lifted function on $Y$.  This equation holds everywhere on $Y$ that $\mu^{a,X}$ is smooth, which is precisely the open dense subset of $Y$ where the $O(n)$ action has finite isotropy.  Our first problem here is that the $R^{a}$ term will be negative when $C^{a}$ is not abelian.  To handle this we begin by looking only at $C^{0}$.  In this case we know that $c^{0}\leq\eta$ is contained in the center of $\eta$ by construction and hence $C^{0}$ is abelian.  This implies $R^{0}=0$ and thus $\triangle_{Y}\mu^{0,X}-\langle\nabla\mu^{0,X},\nabla\mu^{ad}\rangle \leq 0$.  If we could conclude that $\mu^{0,X}$ obtained a minimum somewhere on the nonexceptional part of $Y$ then a maximum principle would conclude that $\mu^{0,X}=constant$.  So let $Y_{E}$ be the exceptional part of $Y$, that is where the $O(n)$ action has non finite isotropy.  But then we know that $Y_{E}$ is precisely where the semidefinite metric $h^{ad,X}$ becomes singular because the projection $p_{O^{\perp}}$ degenerates.  More than that, given the equivariant bundle $FM_{i}\stackrel{f_{i}}{\rightarrow}Y$ we know that an exceptional isotropy orbit in $Y$ corresponds to an orbit in $FM_{i}$ that intersects the center of the nilfibers in $FM_{i}$.  This is precisely the statement that $h^{ad,X}|_{c^{0}}=h^{0,X}$ becomes only semidefinite at these points.  Thus $v^{0,X}=0$ on $Y_{E}$ and so $\mu^{0,X}\rightarrow\infty$ near $Y^{E}$.  Thus by our diameter bound on the $M_{i}$ we have that $Y$ is compact and so $\mu^{0,X}$ obtains a minimum somewhere on $Y-Y_{E}$.  By the maximum principle $\mu^{0,X}=constant$ and hence $|A^{C^{0}}|=0$.

The above allows us to conclude two points.  First $\mu^{0,X}$ is bounded and so $Y_{E}=\emptyset$, hence $X$ has at worst orbifold singularities.  Since $A^{C^{0}}$ vanishes we return to the bundle $U^{X}\times N^{0}\rightarrow U^{X}\times (N^{0}/C^{0})$ to interpret this.  Restrict the metric $g^{U\times N}$ to a single $N^{0}$ fiber.  Let $V^{0},V^{1}$ be horizontal invariant vectors on $N^{0}$ which are perpendicular to $C^{0}\leq N^{0}$ and hence lifts of vectors $\bar V^{0},\bar V^{1}$ on $N^{0}/C^{0}$.  Then $A^{C^{0}}=0$ implies that the projection of the bracket $[V^{0},V^{1}]$ to $C^{0}$ vanishes.  However if $\bar V^{0},\bar V^{1}\in cent(\eta/c^{0})$ as elements of the lie algebra then this implies that $[V^{0},V^{1}]=0$ identically.  In other words $C^{1}$ is also abelian (and in fact also contained in the center of $N^{0}$ as well).  So $R^{1}=0$ and we may repeat the above arguments with $C^{1}$ instead of $C^{0}$, giving us $\mu^{1,A}=constant$.  This process continues inductively until we see that $C^{k}=N$ is abelian with $A^{k}=A^{N}=0$ and $H^{k}=H^{N}=\nabla\mu^{ad,X}=0$.  In particular the submersion $U^{X}\times N^{0}\rightarrow U^{X}$ has $N^{0}$ abelian with vanishing mean curvature and integrability tensor.

The last step of the proof is to show the second fundamental form $T^{N}$ of the Riemannian submersion $U^{X}\times N^{0}\rightarrow U^{X}$ vanishes.  For starters Proposition \ref{prop_curv_comp}.3 now tells us that the orbifold Ricci curvature of $X$ is at least nonnegative because $H^{N}$ vanishes.  We already discussed that $X$ is now an orbifold and so by the orbifold splitting theorem we can pass to the orbifold universal cover $\tilde X\approx X^{c}\times \mathds{R}^{l}$ to get an isometric splitting with $X^{c}$ a compact simply connected orbifold.  Since $Y$ is an orbifold bundle over $X$ we may pass to appropriate covers $\tilde V^{ad}\rightarrow\tilde Y\rightarrow \tilde X$.  The vanishing of $A^{N}$ tells us that $\nabla^{ad}=\nabla^{flat}$ is a flat connection and since $\tilde X$ is orbifold simply connected and $\tilde V^{ad}$ is equivariant we can parallel translate a basis $\{\xi^{j}\}$ of $\tilde V^{ad}$.  Proposition \ref{prop_curv_comp}.3 and that $\tilde X$ is Ricci flat in the $\mathds{R}^{l}$ directions tells us that $T^{N}$ vanishes in the $\mathds{R}^{l}$ directions and in particular $h^{ad,X}(\xi^{j},\xi^{j})$ is at most a function of $X^{c}$ ($T^{N}=\nabla h^{ad,X}$ by Lemma \ref{lem_Veq1}.1).  In particular because $X^{c}$ is compact $h^{ad,X}(\xi^{j},\xi^{j})$ obtains a maximum at some point.  By Proposition \ref{prop_curv_comp}.1 $h^{ad,X}(\xi^{j},\xi^{j})$ satisfies the equation $\triangle_{Y}(h^{ad,X}(\xi^{j},\xi^{j})) -\langle\nabla(h^{ad,X}(\xi^{j},\xi^{j})), \nabla\mu^{ad}\rangle +2|T(\xi^{j},\cdot)|^{2}=0$ and thus a maximum principle thus gives us that $h^{ad,X}(\xi^{j},\xi^{j})$ is constant with $T(\xi^{j},\cdot)=0$.  This holds for each $j$ and so we are done.
\end{proof}

and now we prove Theorem \ref{cor_nt2_main1}:

\begin{proof}[Proof of Theorem \ref{cor_nt2_main1}]

The proof is by contradiction.  Assume for some $n$ and $K$ that no such $\epsilon$ exists.  Then we can find a sequence of complete Riemannian manifolds $(M^{n}_{i},g_{i})$ with $diam=1$, $|sec_{i}|\leq K$ and $|Rc_{i}|\rightarrow 0$ that do not satisfy the statement of the corollary.  But after passing to a subsequence we can apply Theorem \ref{thm_nt2_main2} to see that $(M_{i},g_{i})\rightarrow (X,d)$ where $X$ is a Ricci flat Riemannian orbifold.  Now standard theory tells us that for $i$ sufficiently large that the $M_{i}$ is a singular bundle over $X$ with infranil fibers, we need only check that it is an orbifold bundle.  For that consider the sequence $(FM_{i},g^{FM}_{i},O(n)) \stackrel{eGH}{\rightarrow}(Y,g^{Y},O(n))$ where for $i$ sufficiently large the $FM_{i}$ are equivariant fiber bundles over $Y$.  If we knew that the $O(n)$ action on $Y$ had no exceptional isotropy then the induced bundle $M_{i}\rightarrow X$ would be orbifold as claimed.  However this follows immediately, and in fact was explicitly stated, in the proof of Theorem \ref{thm_nt2_main2} above.

\end{proof}

\section{Directions for Future Work}

We give a brief account of some open questions and directions.

We saw that Theorem \ref{thm_nt2_main2} does not hold without the diameter assumption, our first question is:

\begin{problem}
Under what additional assumptions does Theorem \ref{thm_nt2_main2} hold if the diameter assumption is dropped?
\end{problem}

This could have useful consequences for understanding singularity dilations.

Also when it comes to Theorem \ref{thm_nt2_main2} it seems to the authors that the Ricci pinching condition may be replaced by other pinching conditions.  For instance if $W$ is the Weyl tensor then:

\begin{problem}
Does Theorem \ref{thm_nt2_main2} hold if $|Rc_{i}|\rightarrow 0$ and $X$ being a Ricci flat orbifold is replaced by $|W_{i}|\rightarrow 0$ and $X$ being conformally flat, respectively?
\end{problem}

The conclusion of Theorem \ref{thm_nt2_main2} gives us restrictions not just on the geometry of the limit $X$ but on the topology as well (namely $X$ can have at worst orbifold singularities).

\begin{problem}
Under what more general hypothesis can we restrict the singularity behavior of limits $X$ of manifolds with bounded curvature?
\end{problem}

For instance it follows from the proof of Theorem \ref{thm_nt2_main2} that if we had just assumed only the upper pinching bound $Rc_{i}\leq \epsilon_{i}\rightarrow 0$ then $X$ would still have at worst orbifold singularities.

\appendix

\section{Infranil Bundle Structure}\label{sec_inf_bund}

This section is mainly to review and classify some structure introduced from \cite{F2} and \cite{CFG} and prove some refinements.  We will also introduce an assortment of terminology which is used throughout the paper.  We use the notation that if $\mathcal{V}\rightarrow M$ is a vector bundle then $\Gamma(\mathcal{V})$ denotes the space of smooth sections.  Similarly $\Gamma(M)$ will on occasion be used to denote the smooth functions on $M$.

\begin{definition}
Let $M$ be a smooth manifold with $\mathcal{V}\rightarrow M$ a vector bundle over $M$.  We call $\nabla^{\mathcal{V}}:\Gamma(\mathcal{V})\times\Gamma(\mathcal{V})\rightarrow \Gamma(\mathcal{V})$ a $\mathcal{V}$-connection on $M$ if $\forall U,V,W\in \Gamma(\mathcal{V})$ and $\kappa\in\Gamma(M)$ we have
\begin{enumerate}
\item $\nabla^{\mathcal{V}}_{U+kV}(W) = \nabla^{\mathcal{V}}_{U}(W)+\kappa\nabla^{\mathcal{V}}_{V}(W) \in \Gamma(\mathcal{V})$

\item $\nabla^{\mathcal{V}}_{U}(\kappa V) = d_{U}\kappa V + \kappa\nabla^{\mathcal{V}}_{U}(V)$
\end{enumerate}
\end{definition}

Notice that if $\mathcal{V}\subseteq TM$ is an integrable subbundle then $\nabla^{\mathcal{V}}$ defines an affine connection on its restriction to each invariant submanifold.  In practice we will be interested when $\mathcal{V}$ arises as the vertical subspace induced from a submersion $f:M_{0}\rightarrow M_{1}$ and will write $\nabla^{f}$ in such cases.

In the case $\mathcal{V}$ is induced from a submersion $f:M_{0}\rightarrow M_{1}$ we have that the restriction of $\nabla^{f}$ to each level set of $f$ is an affine connection.  Given any metric $g_{0}$ on $M_{0}$ there is a canonical $\mathcal{V}$-connection induced by $g_{0}$ by projecting the Levi-Civita connection to the vertical distribution, we will refer to this connection as $\nabla^{f,LC}$.  If $\nabla^{f}$ and $\tilde\nabla^{f}$ are any two fiber connections then we see that $\nabla^{f}-\tilde\nabla^{f}$ is tensorial on $\mathcal{V}$.

\begin{definition}
We call a $\mathcal{V}$-connection $\nabla^{f}$ a group $\mathcal{V}$-connection if $\forall y\in M_{1}$ the restriction of $\nabla^{f}$ to $f^{-1}(y)$ is a flat affine connection with parallel torsion.  We call $\nabla^{f}$ an infranil connection if additionally the induced lie algebra is nilpotent, and we call $\nabla^{f}$ a nil connection if we even further assume that the induced holonomy on each $f^{-1}(y)$ is trivial.  We call the pair $(f,\nabla^{f})$ an infranil (resp. nil) bundle.  If $M_{0}$ and $M_{1}$ are Riemannian we say $(f,\nabla^{f})$ is $\{A\}^{k+2,\alpha}_{0}$-regular if $f$ is $\{A\}^{k+2,\alpha}_{0}$-regular and $\nabla^{f}-\nabla^{f,LC}$ is $\{A\}^{k,\alpha}_{0}$-regular.
\end{definition}

Notice that if $\nabla^{f}$ is a group $\mathcal{V}$-connection then it defines a group structure $G$ on the universal cover of each fiber such that the fundamental group $\Lambda$ lies naturally as a discrete subgroup of $G\rtimes Aut(G)$.  If $\nabla^{f}$ is an infranil connection then this group $G\equiv N$ is nilpotent, and if $\nabla^{f}$ is a nil connection then $\Lambda \leq N$ has no automorphism part.  We point out also that if $U$ is any manifold, $N$ a nilpotent lie group and $f:U\times N\rightarrow U$ is the projection map then there is a canonical nil $\mathcal{V}$-connection $\nabla^{N}$ on $U\times N$ which is the defined on each $N$ fiber as being the unique connection for which the left invariant vectors are parallel.  The following uses definitions from Section \ref{sec_notation}.

%Notice the integrability of the distribution $\mathcal{V}$ allows us to use $\nabla^{f}$ to define an exponential map $exp_{\mathcal{V}}:\mathcal{V}\rightarrow M_{0}$ in the usual way.  If $(M_{i},g_{i})$ and $(f,\nabla^{f})$ are $\{A\}^{\infty}_{0}$-regular and we let $V_{\epsilon}=\{V\in\mathcal{V}: ||V||\leq\epsilon\}$ then by writing in local weak coordinates it is easy to check that we can pick $\epsilon=\epsilon(n,A)$ small so that $exp_{\mathcal{V}}|_{\mathcal{V}_{\epsilon}}$ is $\{B\}=\{B(n,A)\}^{\infty}_{0}$-regular and even that $\frac{1}{2}||V||\leq d(exp_{\mathcal{V}}(0\cdot V),exp_{\mathcal{V}}(V))\leq 2||V||$ holds for each $V\in\mathcal{V}_{\epsilon}$.  If $\nabla^{f}$ is a group $\mathcal{V}$-connection then this exponential map is just the usual Lie Group exponential map on each fiber.

\begin{definition}
Let $M_{0}$ and $M_{1}$ be manifolds with $(f,\nabla^{f})$ an infranil bundle with nilpotent structure group $N$.  Let $\varphi:B_{r}(0)\times N \rightarrow M_{0}$ be a weak nilpotent coordinate system.  Then we say that $\varphi$ and $f$ are compatible if the fibers of the lifted map $\tilde f:B_{r}\times N\rightarrow M_{1}$ are the $N$ factors, and the pullback $\mathcal{V}$-connection $\varphi^{*}\nabla^{f}$ is equal to the canonical nil connection $\nabla^{N}$ on $B_{r}\times N$.  If $(M_{0},g_{0})$ and $(M_{1},g_{1})$ are Riemannian and $(N,h)$ is normalized then we say the compatible weak nilpotent map $\varphi$ is $\{A\}^{k+2,\alpha}_{0}$-regular if $(f,\nabla^{f})$ is $\{A\}^{k+2,\alpha}_{0}$-regular and as a Riemannian map $\varphi$ is $\{A\}^{k+2,\alpha}_{0}$-regular.
\end{definition}

As expected if $M_{0},M_{1}$ are $G$ manifolds then we call $(f,\nabla^{f})$ a $G$-infranil bundle if $f$ is $G$ equivariant and the induced action of $G$ on the level sets of $f$ are by affine transformations with respect to $\nabla^{f}$.

As in \cite{CFG} we get as a consequence of Malcev Rigidity the existence of compatible nilpotent coordinates with any infranil bundle with $C^{\infty}$-bounds.  Recall a submersion $f$ between Riemannian manifolds is an $\epsilon$-submersion if for every horizontal vector $X$ we have $1-\epsilon\leq |df[X]| \leq 1+\epsilon$.

\begin{theorem}
Let $(M^{n}_{0},g_{0},p_{0})$ and $(M_{1},g_{1},p_{1})$ be $\{A\}^{\infty}_{0}$-regular at $p_{i}$ with $(f,\nabla^{f})$ an $\{A\}^{\infty}_{0}$-regular $G$-infranil bundle such that $f(p_{0})=p_{1}$, $diam f^{-1}(y) \leq 1$ $\forall y$ and with $f$ a $\frac{1}{2}$-submersion.  Let $\varphi:B_{r}\rightarrow M_{1}$ be a $\{A\}^{\infty}_{0}$-regular coordinate system with $\varphi(0)=p_{1}$.  Then if $r\leq r(n,A)$ then there exists a simply connected normalized nilpotent $(N,h)$ and a $\{B\}^{\infty}_{0}$=$\{B(n,A)\}^{\infty}_{0}$-regular weak nilpotent coordinate system $\tilde\varphi:B_{r}\times N\rightarrow M_{0}$ which is compatible with $(f,\nabla^{f})$ such that $f(\varphi(x,n))=x$.
\end{theorem}

Though we will not reprove the above carefully, for convenience we mention the basic points of the proof.  The connection $\nabla^{f}$ first allows us to identify a fiber $f^{-1}(\varphi(0))$ with $N/\Lambda$ naturally (though not uniquely).  Then we can use the normal exponential map to give us a uniformly bounded diffeomorphism $\bar\varphi:B_{r}\times (N/\Lambda)\rightarrow f^{-1}(\varphi(B_{r}))$ for $r$ sufficiently small. The $\mathcal{V}$-connection $\nabla^{f}$ may not be trivial in these coordinates however, so we perturb again using the Malcev Rigidity which guarantees that the induced mapping $f^{-1}(\varphi(x))\rightarrow f^{-1}(\varphi(0))$ from $\bar\varphi$ induces a canonical affine transformation $f^{-1}(\varphi(x))\rightarrow f^{-1}(\varphi(0))$ $\forall x$.  Notice that without this rigidity, which is a special property of the nilpotency of $N$, that such coordinates simply need not exist.  The need for the \textit{apriori} $C^{\infty}$ bound assumption comes from that although we control regularity at each step, some steps (like the use of exponential coordinates) use higher degrees of regularity than are strictly necessary from a previous step to control lower degrees of regularity in the next step.  One may try to fix this by some form of normal harmonic coordinates or some such methods, but we find that the smoothing lemma \ref{lem_smoothing} fixes this problem in a simpler manner.

With good local coordinates guaranteed above we wish to write down good global conditions.  We will actually do this in two steps.  The first is as follows.

\begin{definition}
Let $M_{0}$ and $M_{1}$ be smooth manifolds.  We say $\{(U_{\alpha}\times N,\Lambda,\varphi_{\alpha})\}$ is an unreduced infranil atlas if $\{U_{\alpha}\}$ is a covering of $M_{1}$ with $\varphi_{\alpha}:U_{\alpha}\times N\rightarrow M_{0}$ weak nilpotent coordinate systems such that $\pi_{\varphi_{\alpha}} = \Lambda$ is independent of $\alpha$ and such that the induced maps $\varphi_{\alpha}:U_{\alpha}\times (N/\Lambda)\rightarrow M_{0}$ give $M_{0}$ a bundle structure over $M_{1}$ whose transition functions are affine transformations.  We say the atlas is $\{A\}^{k+2,\alpha}_{0}$-regular if there exists normalized metrics $h^{coord}_{\alpha}$ such that the local diffeomorphisms $\phi_{\alpha}$ become $\{A\}^{k+2,\alpha}_{0}$-regular.
\end{definition}

In particular an (unreduced) infranil atlas naturally defines an infranil bundle structure and we say an infranil atlas is compatible with a given infranil bundle $(f,\nabla^{f})$ if the induced bundle structure is equal to $(f,\nabla^{f})$.  If $M_{0}$ and $M_{1}$ are $G$-manifolds then we say the infranil atlas is $G$-equivariant if the induced submersion map $f$ is $G$-equivariant and the induced mapping on the fiber connection $\nabla^{f}$ is an affine isometry.  As expected we call the $G$ action $\{A\}^{\infty}_{0}$-regular if it is so bounded in the charts belonging to the infranil atlas.  So using the last theorem we can immediately get

\begin{theorem}\label{thm_infra_atlas}
Let $(M^{n}_{0},g_{0},p_{0})$ and $(M_{1},g_{1},p_{1})$ be $\{A\}^{\infty}_{0}$-regular at $p_{i}$ with $inj(M_{1})\geq\iota>0$ and $(f,\nabla^{f})$ an $\{A\}^{\infty}_{0}$-regular $G$-infranil bundle such that $f(p_{0})=p_{1}$, $diam f^{-1}(y) \leq 1$ $\forall y$ and with $f$ a $\frac{1}{2}$-submersion.  If $r\leq r(n,A,\iota)$ then there exists $\{B\}^{\infty}_{0}$=$\{B(n,A,r)\}^{\infty}_{0}$ such that for any cover $\{U_{\alpha}\}=\{B_{r}(x_{\alpha})\}$ of $M_{1}$ we can construct a $\{B\}^{\infty}_{0}$-regular unreduced infranil atlas $\{(U_{\alpha}\times N,\Lambda,\varphi_{\alpha})\}$ which is compatible with $(f,\nabla^{f})$.
\end{theorem}

Our reason for calling the infranil atlas above unreduced is the following.  If $N$ and $\Lambda$ are as before then the Lie Group of affine transformations on $N/\Lambda$ is $\frac{N}{Cent(N)\cap\Lambda}\rtimes Aut(\Lambda)$.  The above theorem is the statement that the fiber bundle $f:M_{0}\rightarrow M_{1}$ has its structure group reduced to $Aff(N/\Lambda)$.  However if we consider the subgroup $C_{Aff}\equiv \frac{Cent(N)}{Cent(N)\cap\Lambda}\rtimes Aut(\Lambda)$ then we see that $Aff(N/\Lambda)/C_{Aff}$ is contractible.  Hence in principal we should be able to reduce the structure group of $f:M_{0}\rightarrow M_{1}$ to $C_{Aff}$ and the only thing holding us back is checking that we can do this while keeping good regularity of our coordinates.  In fact this is not so hard and so we call an unreduced infranil atlas $\{(U_{\alpha}\times N,\Lambda,\varphi_{\alpha})\}$ simply an infranil atlas if the coordinate transformations lie in $C_{Aff}$.  In general we call actions by $C_{Aff}$ on $N/\Lambda$ central affine transformations.  We then get the following:

\begin{theorem}\label{thm_infra_atlas2}
Let $(M^{n}_{0},g_{0},p_{0})$ and $(M_{1},g_{1},p_{1})$ be as in Theorem \ref{thm_infra_atlas}.  Then for any $r\leq r(n,A,\iota)$ there there exists $\{B\}^{\infty}_{0}$=$\{B(n,A,r)\}^{\infty}_{0}$ and a cover $\{U_{\alpha}\}=\{B_{r}(x_{\alpha})\}$ of $M_{1}$ such that we can construct a $\{B\}^{\infty}_{0}$-regular infranil atlas $\{(U_{\alpha}\times N,\Lambda,\varphi_{\alpha})\}$ which is compatible with $(f,\nabla^{f})$.  Further if the $G$ action is $\{A\}^{\infty}_{0}$-regular then we can take the group $G$ to act by central affine automorphisms in the infranil atlas charts.
\end{theorem}
\begin{proof}
Let $\{B_{r}(x_{\alpha})\}$ be a covering of $M_{1}$ with $\phi_{\alpha}:M_{1}\rightarrow \mathds{R}$ a $\{C\}^{\infty}_{0}=\{C(n,A,r)\}^{\infty}_{0}$-regular partition of unity.  Let $\{(U_{\alpha}\times N,\Lambda,\varphi'_{\alpha})\}$ be the associated unreduced infranil atlas from Theorem \ref{thm_infra_atlas}.  Associated to this atlas is the $Aff(N/\Lambda)$ principal bundle $P_{Aff}\rightarrow M_{1}$ with local coordinates $U_{\alpha}\times Aff(N/\Lambda)$ whose transition functions are induced from the affine transformations $\varphi'_{\alpha}$.  With $C_{Aff}\leq Aff(N/\Lambda)$ as before we have the fiber bundle $P'\equiv P_{Aff}/C_{Aff}\rightarrow M_{1}$ with fibers $Aff(N/\Lambda)/C_{Aff}\approx N'$ a simply connected nilpotent.  We need to find a global section $s:M_{1}\rightarrow P'$ of this fiber bundle such that in the induced local coordinates $U_{\alpha}\times N'$ the section $s:U_{\alpha}\rightarrow N'$ is $\{C\}_{0}^{\infty}=\{C(n,A,r)\}_{0}^{\infty}$-regular.  Such a global section of course identifies a unique $C_{Aff}$ orbit in $Aff(N/\Lambda)$ for each point of $M_{1}$ and gives rise to our desired reduction.

To find this section we proceed as one might expect.  Note again that $Aff(N/\Lambda)/C_{Aff}\approx N'$ is a simply connected, hence contractible, nilpotent lie group which we equip locally with the quotient metric induced from $h^{coord}_{\alpha}$ on $N$.  Begin by letting $s_{\alpha}:U_{\alpha}\rightarrow Aff(N/\Lambda)/C_{Aff}\approx N'$ be the identity map for each $\alpha$.  If $U_{\alpha}\cap U_{\beta}\neq \emptyset$ then we can let $s_{\alpha\beta}\equiv s_{\alpha}\circ\varphi'_{\alpha}\circ (\varphi'_{\beta})^{-1}:U_{\alpha}\cap U_{\beta}\rightarrow N'$.  Because of our bounds on the transition functions $\varphi'$ each $s_{\alpha\beta}$ is clearly $\{C\}^{\infty}_{0}=\{C(n,A,r)\}^{\infty}_{0}$-regular.  To construct our global section $s$ we need to appropriately average the maps $s_{\alpha\beta}$.  Applying the center of mass technique of \cite{BK} to the functions $s_{\alpha\beta}:U_{\alpha}\rightarrow N'$ with weights $\phi_{\beta}$ gives us such a canonical averaging.  The coordinate transformations $\varphi'$ are affine transformations and the averaging procedure depends only on the affine structure of $N'$, so this procedure gives a well defined section $s:M_{1}\rightarrow P'$.  Because the functions $s_{\alpha\beta}$ and $\phi_{\beta}$ are $\{C\}^{\infty}_{0}$-regular so is the map $s$ and we are done.

If we further assume that the $G$ action is $\{A\}^{\infty}_{0}$-regular then by another center of mass argument we can take $s:M_{0}\rightarrow P'$ to be $G$-equivariant with respect to naturally induced $G$ action on $P'$.  The regularity of the original $s$ map and the $G$ action guarantees the regularity of the averaged map and we are again done.
\end{proof}

A final definition which will be of use is to construction Riemannian infranil bundles where the geometric structures and the bundle structures are related:

\begin{definition}\label{def_sym_bund}
Let $(M_{0},g_{0})$ and $(M_{1},g_{1})$ be Riemannian manifolds and $\{(U_{\alpha}\times N,\Lambda, \varphi_{\alpha}\}$ an infranil atlas.  Then we say $(M_{0},g_{0})$ is compatible with the infranil bundle if the right $N$ actions on the Riemannian manifolds $(U_{\alpha}\times N,\varphi_{\alpha}^{*}g_{0})$ are isometric actions.
\end{definition}

\end{document}